\documentclass[11pt]{amsart}
\usepackage{graphicx, subfigure}
\usepackage{txfonts}

\newtheorem{theorem}{Theorem}[section]
\newtheorem{proposition}[theorem]{Proposition}

\theoremstyle{remark}
\newtheorem{remark}[theorem]{Remark}

\begin{document}
\baselineskip 16pt

\thispagestyle{empty}

\begin{center}\sf
{\Large Numerical exploration of a}\medskip

{\Large forward--backward diffusion equation}\medskip

{\tt \today}\\
Pauline LAFITTE\footnote{EPI SIMPAF - INRIA Lille Nord Europe Research Centre,
Parc de la Haute  Borne, 40, avenue Halley, F-59650 Villeneuve d'Ascq cedex, France,
\& Laboratoire Paul Painlev\'e, UMR 8524, C.N.R.S. --
  Universit\'e de Sciences et Technologies de Lille, Cit\'e Scientifique,
  F-59655 Villeneuve d'Ascq cedex (FRANCE), \texttt{\tiny
    lafitte@math.univ-lille1.fr}}, Corrado MASCIA\footnote{Dipartimento di
  Matematica ``G. Castelnuovo'', Sapienza -- Universit\`a di Roma, P.le Aldo
  Moro, 2 - 00185 Roma (ITALY), \texttt{\tiny mascia@mat.uniroma1.it}}
\end{center}
\vskip.5cm

\begin{quote}\small \baselineskip 14pt 
{\sf Abstract.} We analyze numerically a forward-backward diffusion equation with a cubic-like diffusion function,  
--emerging in the framework of phase transitions modeling-- and its ``entropy'' formulation
determined by considering it as the singular limit of a third-order pseudo-parabolic equation.
Precisely, we propose schemes for both the second and the third order equations, we discuss
the analytical properties of their semi-discrete counterparts and we compare the numerical results 
in the case of initial data of Riemann type, showing strengths and flaws of the two approaches,
the main emphasis being given to the propagation of transition interfaces.
\vskip.15cm

\vskip.15cm

\end{quote}

\baselineskip=16pt

\section{Introduction}

The aim of the present article is to investigate numerically the solutions to a nonlinear 
diffusion equation of the form
\begin{equation}\label{nonlinpar}
 \frac{\partial  u}{\partial t}=\frac{\partial^2 \phi(u)}{\partial x^2}
 \qquad\qquad (x,t)\in\mathbb{R}\times (0,+\infty)
\end{equation}
where $u\,:\,\mathbb{R}\times[0,+\infty)\to\mathbb{R}$, for a  non-monotone diffusion function 
$\phi\,:\,\mathbb{R}\to\mathbb{R}$; specifically, the function $\phi$ is to be cubic-like shaped, 
i.e. monotone decreasing in a given interval $I$ and monotone increasing elsewhere.

Due to the presence of intervals where $\phi$ is a decreasing function, the initial value problem 
for \eqref{nonlinpar} is ill--posed and an appropriate (non classical) definition of solution 
has to be introduced.  
In this respect, a common paradigm states that the loss of a well-posed classical framework may
emerge as consequence of some simplification of a given complete physical model, 
consisting in neglecting some higher order term because of the presence of some
small multiplicative factor $\varepsilon$.
The usual strategy is then to consider classical solutions to the original higher order equation,
pass to the limit as a definite parameter $\varepsilon$ tends to zero and regard such singular 
limit as the ``physical'' solution of the problem under consideration.  
Whenever possible, it is interesting to prescribe a definition of generalized solution to the reduced 
problem that is consistent with the limiting procedure and, at the same time, is determined by 
proper conditions to be verified by such a generalized solution, without explicit reference to the 
disregarded higher order terms.  
In this respect, a classical example is given by scalar conservation laws where solutions can be 
defined through a vanishing-viscosity approach and well-posedness is restored by considering an 
entropy setting.  
Different higher perturbations may, in principle, lead to different limiting formulation, as in the case 
of van der Waals fluids, where viscosity and viscosity-capillarity criteria give raise to different jump 
conditions in the limit,\cite{Slem83}.
In the same spirit, the approach through singular perturbations is commonly used in the context of 
calculus of variations to determine appropriate minima of non-convex energy functionals (among 
others, see the classical reference Ref.~\cite{Modi87}).  
A related approach consists in studying the structure of the global attractors of the higher order 
equations in the singular limit, as it has been performed for a relaxation viscous Cahn--Hilliard 
equation in Ref.~\cite{GGPM05}.

Here, the model \eqref{nonlinpar} arises as a reduced equation in the context of phase transitions  and 
we consider the formulation for \eqref{nonlinpar}, 
established in Refs.~\cite{Plot93}, \cite{Plot94}, \cite{EvaPor04}, 
obtained as singular limit of the third order equation of {\it pseudo-parabolic type}
\begin{equation}\label{relaxing}
 \frac{\partial u}{\partial t}= \frac{\partial^2}{\partial x^2} 
  \left(\phi(u)+\varepsilon\,\frac{\partial u}{\partial t}\right)
  \qquad\qquad
  x\in \mathbb{R},\quad t>0,
\end{equation} 
originally considered in Ref.~\cite{NovPeg91}.
The denomination ``pseudo-parabolic'', relative to the presence of the time derivative of the elliptic 
operator $\partial^2/\partial x^2$, has been proposed in Refs.~\cite{ShoTin70}, \cite{Ting69};
equations of the same form are also called of {\it Sobolev type} or {\it Sobolev--Galpern type}.
A sketch of the derivation of such an equation in the phase transition setting is given in  
Section \ref{sec:physical}.

Equations similar to \eqref{relaxing} arise also in modeling fluid flow in
heat conduction and shear in second order fluids (see, respectively,
Ref.~\cite{CheGur68}, Ref.~\cite{CoDuMi65} and descendants).  Besides, the third order
term in \eqref{relaxing} appears in the description of long waves in shallow
water (see Refs.~\cite{BenBonMah72}, \cite{Pere67} and descendants), as an alternative to
the usual purely spatial third order term considered in the Korteweg--DeVries
equation.  Taking the singular limit of such higher order equations gives in
general a different formulation with respect to the classical entropy
formulation obtained by means of the vanishing viscosity limit,\cite{DuPePo07}.  
Regardless of the physical context, the use of the third
order term as a regularizing term for an ill-posed diffusion equation of the
form \eqref{nonlinpar}, is called the {\sf quasi-reversibility method}, and it
has been used by many authors, mainly in the linear case (see Ref.~\cite{Payn75}
and references therein).

Independently of the monotonicity of the function $\phi$, the Cauchy problem
for equation \eqref{relaxing} is well-posed, \cite{NovPeg91}; moreover, the singular limit $u$ 
of the family $\{u^\varepsilon\}$, solutions to \eqref{relaxing}, can be described in term of 
Young measures, \cite{Plot93} \cite{Plot94}. 
The sign of $\phi'$ is relevant for the dynamical properties of the solutions: precisely,
the constant state $\bar u$ turns to be (asymptotically) stable if $\phi'(\bar u)>0$
and unstable if $\phi'(\bar u)<0$.
In the case under consideration, it is possible to distinguish two (unbounded) stable phases
divided by a (bounded) unstable phase, the so-called {\it spinodal region}.
Generically, initial data with values in the spinodal region generate oscillations which lead
to patterns mixing the two stable phases.

 Assuming additionally that the limit $u$ of the family $\{u^\varepsilon\}$ is piecewise smooth and 
 that it takes values only in the regions where the function $\phi$ is monotone increasing, it is
possible to determine appropriate admissibility conditions, called {\sf entropy conditions}, on a function 
in order for it to be limit of solutions to the higher order equation \eqref{relaxing} 
(see Proposition \ref{prop:entrsol} below).  
For such class of generalized solutions, from now on denominated {\sf two-phase solutions}, uniqueness 
and (local) existence have been proved, \cite{MaTeTe09}.

Here our aim is to investigate numerically the behavior of solutions to
\eqref{nonlinpar} exploring two possible strategies: on the one hand, by
discretizing the third-order equation \eqref{relaxing} and considering the
behavior of the algorithm as $\varepsilon\to 0$; on the other hand, by
providing an approximation of the solutions to \eqref{nonlinpar}, taking
advantage of the entropy conditions.  Schematically, if we denote by $h$ the
size of the discrete spatial mesh and by $\varepsilon$ the parameter appearing
in \eqref{relaxing}, both values are assumed to be small, ideally tending to
zero: in the first approach, we let $\varepsilon\to 0^+$ and then ${h}\to
0^+$. Viceversa, in the latter, we take the limits in the opposite order.
Absence of uniformity of one parameter with respect to the other implies that
the two procedures lead, in principle, to different results and we are
interested in comparing them.  Specifically, dealing directly with
\eqref{relaxing} obliges to treat solutions taking values also in the region
where $\phi$ is decreasing, with the consequent appearance of strong
oscillations due to the exponential instability of constant states for the
backward heat equation. 
On the contrary, the two-phase solutions are designed to take values only in the region 
of stability ($\phi$ increasing) and to cross the unstable region in single specific transition points 
where appropriate entropy transmission conditions should be satisfied.

We benchmark the two approaches against Riemann problems, determined by the values 
$(u^-,u^+)$  with $u^\pm$ belonging to the two different stable phases.
Depending on the location of such values, the initial phase transition discontinuity, hereafter 
called  {\it interface} in reference to the porous media equation \cite{minato}, may or may not 
move  (see Section \ref{sec:physical}).
As expected, relevant differences appear in the location of phase boundary depending on the 
algorithm chosen (see Figure \ref{fig:error_zeta}, and Section \ref{sec:twophase}).
\begin{figure}[ht]\label{fig:error_zeta}\centering
\includegraphics[width=9cm,height=6cm]{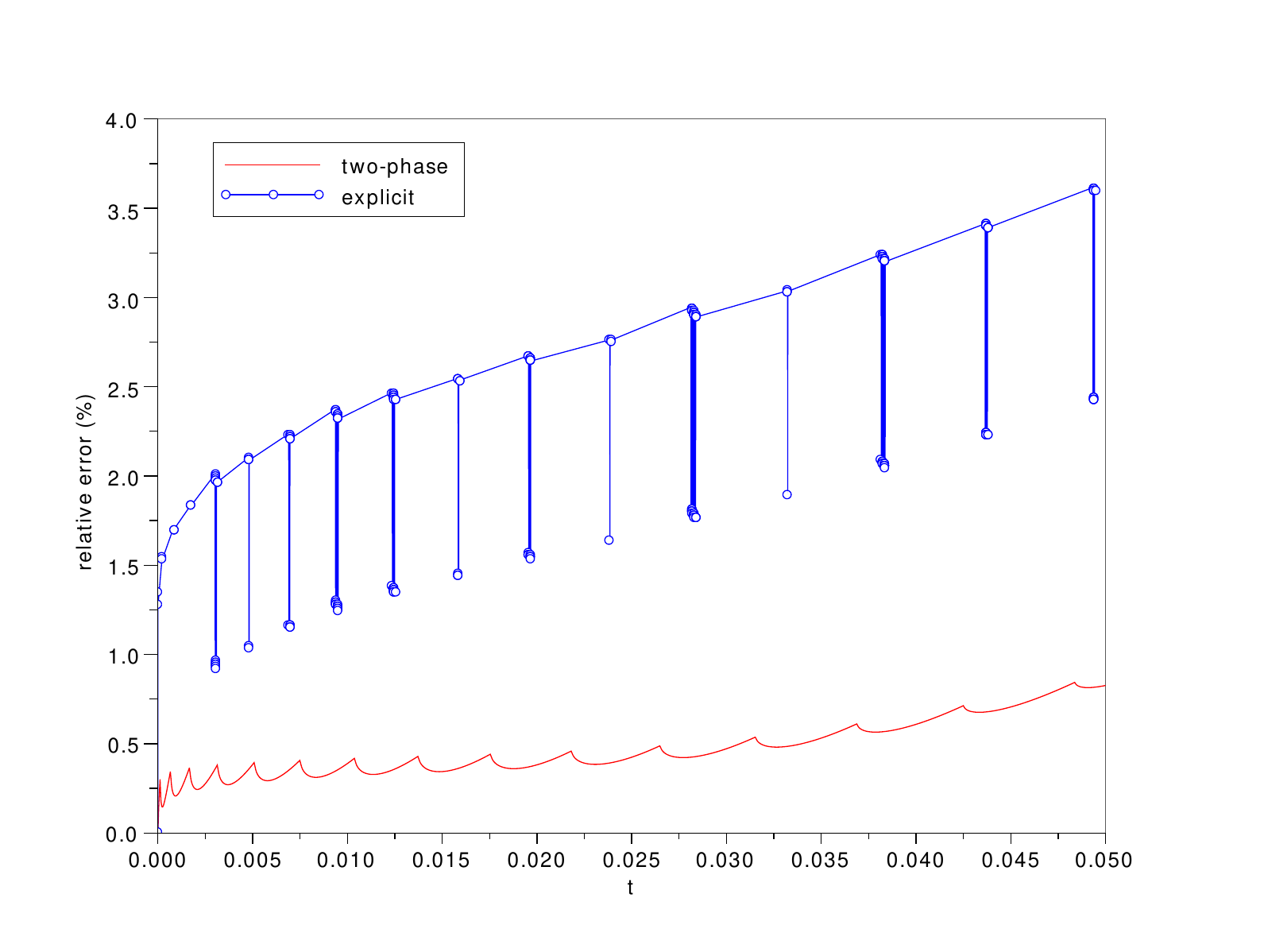}

\caption{\small Evolution of the relative error of the position of the interface for an initial data of
Riemann type with $(u^-,u^+)=(-2,4)$ for the two approaches: explicit scheme for the 
pseudo-parabolic equation (blue), two phase scheme (red).}
\end{figure}

Coming back to the original modeling, it would be more appropriate to consider equation 
\eqref{nonlinpar} as the singular limit of some higher order equation with a much more 
complicated structure with respect to the one of \eqref{relaxing} (see Section \ref{sec:physical}). 
To our knowledge, in this case there is still no available limiting formulation, meaning that
one does not know which are the admissibility conditions for the solutions to \eqref{nonlinpar} to
be singular limits of solutions to the corresponding higher order equation.
Hence, the only possible approach consists in taking the singular limit after having
discretized the complete equation.
Such procedure can not be rigorously justified because of the non-uniform dependence of
$\varepsilon$ with respect to ${h}$.
Therefore, analyzing a simplified situation where the two approaches are both available and can 
be compared, can be heuristically considered as a reliability test on the exchange of the 
limits $\varepsilon\to 0^+$ and ${h}\to 0^+$.   
    
The article is organized as follows.  
In Section \ref{sec:physical}, we sketch the derivation of the modeling equation in order to 
identify the regime of parameters that leads to \eqref{relaxing} and \eqref{nonlinpar}.  
Also, we recall the main results on the analytical properties of the equation, with particular 
care to the two-phase formulation.  
We also deduce an explicit formula for the solution of the Riemann problem in the
case of a piecewise linear diffusion function $\phi$.
In Section \ref{sec:relaxing}, we propose a standard finite-difference algorithm for the third order 
pseudo-parabolic equation \eqref{relaxing}, as considered in Ref.~\cite{EvaPor04}, and 
we analyze the qualitative behavior of the solutions for the corresponding semi-discrete scheme.
In Section \ref{sec:numerelaxing}, we test the fully-discretized algorithm in the case of  Riemann type 
initial data, showing the emergence of incorrect oscillations in the solution.  
Finally, in Section \ref{sec:twophase}, we implement an algorithm for the two-phase solution of
\eqref{nonlinpar}, consisting in coupling finite-differences schemes for the two stable phases 
with  some appropriate transmission conditions at the phase boundary and we compare the results
given by the two approaches.

\section{Physical and analytical background}\label{sec:physical}

Phase transitions modeling is a wide area of active research in pure and applied mathematics.  
The basic target is to describe the behavior of a mixture with many different (stable or unstable) phases, 
with special attention to the dynamics of separating interfaces and pattern formation.
Various approaches and models have been proposed and analyzed, depending on the specific context.  
Probably, the most celebrated is the Cahn--Hilliard model, which is based on defining an appropriate cost 
functional which determines two preferred states and, at the same time, penalizes inhomogeneities 
by means of a gradient term (for a survey on Cahn-Hilliard and phase field models, see Ref.~\cite{Fife00}).

Here, we consider a different model, devised in Ref.~\cite{NovPeg91}, that still preserves the basic 
structure of two stable ``competing'' phases and gain regularity from memory effects, without taking into 
account the cost of transition layers.  
In this Section, we present a simplified version of the derivation of the model equation, and then 
we recall the main known analytical results.

\subsection*{Derivation of the model}
Consider a generalized diffusion equation, in a one-dimensional domain described by 
the space variable $x\in\mathbb{R}$,
 \begin{equation}\label{gendiff}
 \frac{\partial u}{\partial t} = \frac{\partial^2 v}{\partial x^2} 
 \qquad\qquad x\in\mathbb{R},\, t>0,
\end{equation}
where $u=u(x,t)$ and $v=v(x,t)$ are real valued functions.
Following Ref.~\cite{JacFri85}, the unknown $v$, called {\it potential}, is supposed 
to satisfy the relation
\begin{equation}\label{relaxdiff}
  v:=\psi(u)+\int_{-\infty}^t \theta'(t-s)\,\bigl(\psi(u)-\phi(u)\bigr)(x,s)\,ds
\end{equation}
or, equivalently, 
\begin{equation*}
 v=\phi(u)+\int_{-\infty}^t \theta(t-s)\,\frac{\partial}{\partial s}
  \bigl(\psi(u)-\phi(u)\bigr)(x,s)\,ds,
\end{equation*}
where $\theta$ is a {\it relaxation  function}, assumed to be monotone decreasing
and such that $\theta(0)=1$ and $\theta(+\infty)=0$.  
The integral term in the definition of $v$ represents memory effects localized in space, 
with decay described by the relaxation function $\theta$.  
The case with no-memory is obtained by choosing $-\theta'$ equal to the usual Dirac 
distribution concentrated at zero.  
In this case, there holds $v=\phi(u)$ and equation \eqref{gendiff} reduces to a
standard nonlinear diffusion equation.

The relaxing structure \eqref{relaxdiff} mimes the analogous expression used
as a stress--strain law in the modeling of viscoelastic media and based on the
Boltzmann superposition principle (for the general mathematical theory, see
Ref.~\cite{RenHruNoh87} for viscoelastic materials with memory and Ref.~\cite{Rena00}
for viscoelastic flows).  In that context, functions $\psi$ and $\phi$
represent, respectively, the instantaneous elastic and the equilibrium
stress--strain responses.  Similar interpretations can be given also in the
present context as a contribution to the potential $v$.

The simplest choice for $\theta=\theta(t)$ is a single exponential function $e^{-t/\tau}$, $\tau>0$.  
In this case, it is possible to write a partial differential equation for the unknown $u$.  
Indeed, since $\theta(0)=1$ and $\theta'=-\theta/\tau$, there holds
\begin{equation*}
  \frac{\partial v}{\partial t}
  =\frac{\partial \psi(u)}{\partial t}
  -\frac{1}{\tau}\int_{-\infty}^t \theta(t-s)\,\frac{\partial}{\partial s}
  \bigl(\psi(u)-\phi(u)\bigr)(x,s)\,ds
  = \frac{\partial \psi(u)}{\partial t}-\frac{1}{\tau}(v-\phi(u))
\end{equation*}
Therefore, a relaxation function of exponential type corresponds to the assumption that the 
potential $v$ solves a relaxation-type equation
\begin{equation*}
 \frac{\partial}{\partial t}\left(v-\psi(u)\right)
  =\frac{1}{\tau}\left(\phi(u)-v\right)
\end{equation*}
The case with no-memory effects is obtained by taking the singular limit $\tau\to 0$ and, 
again, formally setting $v=\phi(u)$.
Substituting in \eqref{gendiff} differentiated with respect to $t$, we infer
\begin{equation*}
 \frac{\partial^2 u}{\partial t^2} = \frac{\partial^2}{\partial x^2} 
  \left(\frac{\partial \psi(u)}{\partial t}+\frac{1}{\tau}\phi(u)\right)
  -\frac{1}{\tau}\frac{\partial^2 v}{\partial x^2};
\end{equation*}
hence, using \eqref{gendiff} to eliminate the second derivative of $v$,
we obtain
\begin{equation*}
 \tau\,\frac{\partial^2 u}{\partial t^2}
  +\frac{\partial  u}{\partial t}=\frac{\partial^2}{\partial x^2}
  \left(\phi(u)+\tau\,\frac{\partial \psi(u)}{\partial t}\right)
\end{equation*}
Incorporating in the analysis a van der Waals term to take into account the free-energy 
cost of inhomogeneities would lead to a relaxation viscous perturbation of the 
classical fourth-order Cahn--Hilliard equation (containing also a fifth-order term
arising from the coupling of the inhomogeneity cost term and the memory effect).

From now on, we choose $\psi$ in the form $\psi(u)=\kappa\,u$ with $\kappa>0$. 
Hence, the unknown $u$ satisfies the equation
\begin{equation}\label{hyperbolic}
 \tau\,\frac{\partial^2 u}{\partial t^2}
  +\frac{\partial  u}{\partial t}=\frac{\partial^2}{\partial x^2}
  \left(\phi(u)+\varepsilon\,\frac{\partial u}{\partial t}\right)
\end{equation}
where $\varepsilon:=\tau\,\kappa$.  
In Ref.~\cite{NovPeg91}, the regime $\tau\to 0^+$ with $\varepsilon$ being fixed 
has been considered, yielding the pseudo-parabolic equation \eqref{relaxing}. 
The nonlinear diffusion equation \eqref{nonlinpar} arises in the limiting 
regime $\varepsilon\to 0^+$ of equation \eqref{relaxing}.  
Equivalently, we can consider the system of partial differential equations
\begin{equation*}
 \left\{\begin{aligned}
  &\frac{\partial u}{\partial t}-\frac{\partial w}{\partial x}=0,\\
  & \tau\,\frac{\partial w}{\partial t}-\frac{\partial \phi(u)}{\partial x}
   =\varepsilon\,\frac{\partial^2 w}{\partial x^2}-w
 \end{aligned}\right.
\end{equation*}
where $w$ denotes the space-derivative $\dfrac{\partial v}{\partial x}$ of the potential $v$,
and then consider equations \eqref{relaxing} and \eqref{nonlinpar} as formally
obtained by passing to the limit (in this order) $\tau\to 0^+$ and $\varepsilon\to 0^+$.

Let us stress that the assumptions on the relative sizes of parameters $\tau$ and 
$\varepsilon$ considered is not motivated by any physical argument and, thus,  is 
a pure mathematical simplification. 
To our knowledge, a complete description of the limiting behavior of the equation
\eqref{hyperbolic} as $\tau, \varepsilon\to 0^+$ for different relative sizes of the two
parameters is still not available.

\subsection*{Overview on the analytical theory}
Next, let us review the main analytical results known about \eqref{nonlinpar}
and \eqref{relaxing} and their relations.
Let $\phi\,:\,\mathbb{R}\to\mathbb{R}$  be a locally Lipschitz continuous function with a cubic-like 
graph with a local minimum with value $A$ and a local maximum with value $B$, 
for some $A<B$: precisely, we assume
\begin{equation*}
 \phi\textrm{ strictly increasing in }(-\infty,b)\cup(a,+\infty),\qquad
  \phi\textrm{ strictly decreasing in }(b,a)
\end{equation*}
for some $b<a$  with $\phi(a)=A$ and $\phi(b)=B$.
For later reference, we assume the existence of $c\in(-\infty,b)$ and $d\in(a,+\infty)$
such that $\phi(c)=A$ and $\phi(d)=B$.

The monotonicity of the function $\phi$ distinguishes three different phases
\begin{equation}\label{phases}
  \textrm{phase } \mathcal{S}^-=(-\infty,b],\qquad
  \textrm{phase } \mathcal{U}=(b,a),\qquad
  \textrm{phase } \mathcal{S}^+=[a,\infty).
\end{equation}
Being interested in the dynamics generated by the different stability properties
of the regions $(-\infty,b)$, $(b,a)$ and $(a,+\infty)$, we will restrict the analysis
to the case of a piecewise linear function and, specifically, to the case
\begin{equation}\label{pwlinsymm}
  \phi(u)=2u+\frac{3}{2}\bigl(|1-u|-|1+u|\bigr).
\end{equation}
In this case $c=-2, b=-1, a=1, d=2$  and $A=-1, B=1$ (see Figure \ref{fig:phi}).
\begin{figure}[ht]\centering
\includegraphics[height=5cm, width=6cm]{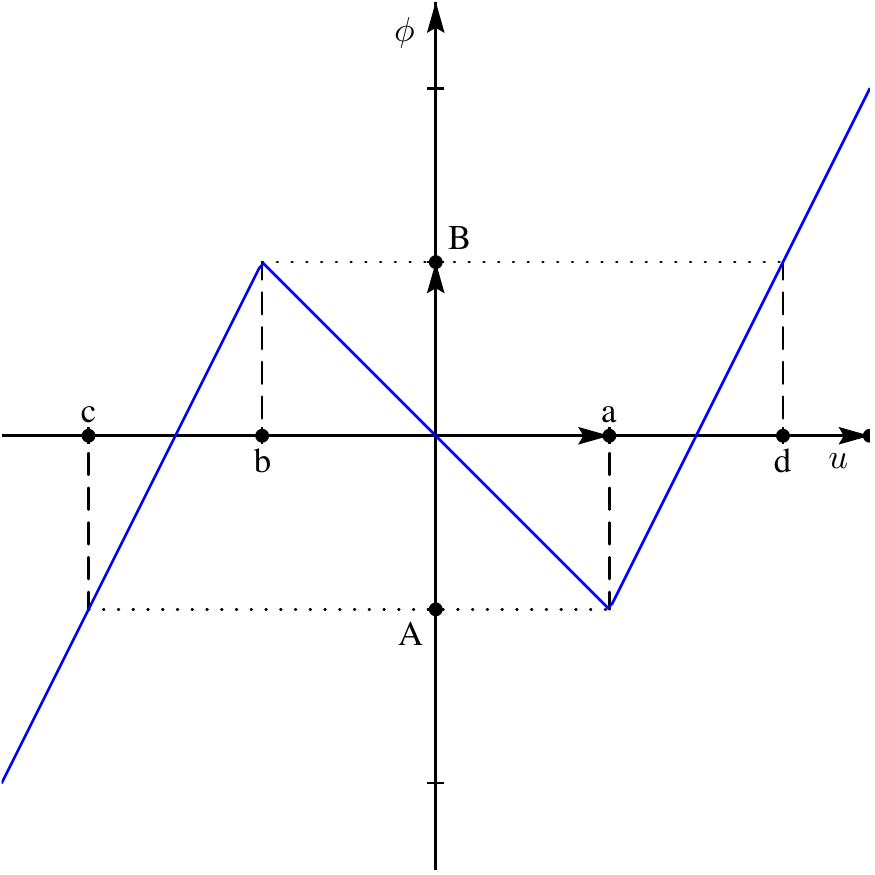}

\caption{\small\footnotesize The piecewise linear diffusion function $\phi$.}
\label{fig:phi}
\end{figure}
The third-order term in \eqref{relaxing} acts as a regularizing term and the initial value problem 
on a bounded interval $I$, with Neumann boundary conditions, for the pseudo-parabolic 
equation \eqref{relaxing} is  well posed in the classical  sense for initial data either bounded 
or continuous, \cite{NovPeg91} \cite{Padr04}.
In Ref.~\cite{NovPeg91}, it is also shown that any bounded measurable function $u=u(x)$ 
such that $\phi(u)$ is constant is a stationary solution of the equation and, if, in addition,
$\phi'(u(x))>0$ for any $x$, such a solution is also asymptotically stable (with exponential 
decay rate) with respect to zero-mass initial perturbations.
In particular, given $u^-<b<a<u^+$ and $x_0\in I$, step functions
\begin{equation}\label{riemannshape}
 u(x)=\left\{\begin{aligned}
  &u^- \qquad &x<x_0,\\ &u^+ \qquad &x>x_0,
  \end{aligned}\right.
\end{equation}
satisfying the transmission condition
\begin{equation*}
	\phi(u^-)=\phi(u^+)\in(A,B) 
\end{equation*}
are attracting time-independent solutions to \eqref{relaxing}. 
In the limit $\varepsilon\to 0^+$ such functions  persist to be solution of the 
corresponding limiting equation \eqref{nonlinpar}.  

In Refs.~\cite{Plot93}, \cite{Plot94}, a deep analysis of the limit $\varepsilon\to 0^+$ 
of the solutions $u^\varepsilon$ to the pseudo-parabolic equation \eqref{relaxing} has been 
performed within the Young measure framework.
Poorly speaking, the analysis of Plotnikov shows that the singular limit of the family of 
solutions $u^\varepsilon$ can be represented as 
\begin{equation}\label{plotnikov}
 u=\lambda^-(\phi^-)^{-1}(v)+\lambda^0 (\phi^0)^{-1}(v)
  +\lambda^+ (\phi^+)^{-1}(v)
\end{equation}
where $\lambda^0, \lambda^\pm$ and $v$ are bounded functions and $\phi^-,
\phi^0, \phi^+$ represent the restriction of $\phi$ in $(-\infty, b)$, $[b,a]$
and $[a,+\infty)$, respectively.  Functions $\lambda^0, \lambda^\pm$ are
non-negative and satisfy
\begin{equation*}
 \lambda^-(x,t)+\lambda^0(x,t)+\lambda^+(x,t)=1,\quad
	\textrm{and}\qquad
 \left\{\begin{aligned}
  &\lambda^-(x,t)=1	&\quad\textrm{if}\quad &v(x,t)<A,\\
  &\lambda^+(x,t)=1	&\quad\textrm{if}\quad &v(x,t)>B,
  \end{aligned}\right.
\end{equation*}
for any $(x,t)$ under consideration.  Finally, for any non decreasing function
$g \in C^1(\mathbb{R})$, setting $G^*:=\lambda^-\,G((\phi^-)^{-1}(v))+\lambda^0
G((\phi^0)^{-1}(v)) +\lambda^+ G((\phi^+)^{-1}(v))$, where
\begin{equation*}
 G(s):=\int_0^s g(\phi(\sigma))\,d\sigma,
 \end{equation*}
the following inequality holds in the sense of distribution
\begin{equation}\label{entropyPlot}
  \frac{\partial G^*}{\partial t} -\frac{\partial}{\partial x}
  \left(g(v) \frac{\partial v}{\partial x}\right)
  + g'(v)\left(\frac{\partial v}{\partial x}\right)^2\leq \, 0 
\end{equation}
Representation \eqref{plotnikov} can be heuristically interpreted by stating that, in the 
limit $\varepsilon\to 0$, the solution $u$ is decomposed in the sum of three contributions, 
each relative to one of the three phases $\mathcal{S}^\pm$ and $\mathcal{U}$.
Functions $\lambda^0, \lambda^\pm$ represent the fraction of $u$ in the 
corresponding  phase.  
Inequality \eqref{entropyPlot} can be interpreted as an entropy inequality, recording
the irreversibility of the limiting process.  
At the present time, it is not known if all these conditions are sufficient for determining uniqueness 
of the Cauchy problem (results on large-time behavior are obtained in Ref.~\cite{SmarTese10}).
Hence, the previous list of conditions is unsatisfactory for determining a well--posed 
limiting dynamics.

\subsection*{Two--phase solutions}
Linearization around an unstable value $\bar u\in(b,a)$ readily shows that solutions of the 
third order equation \eqref{relaxing} tend to escape from the unstable interval exponentially fast,
precisely with rate $\exp(-\phi'(\bar u)\,t/\varepsilon)$.
In the limit $\varepsilon\to 0^+$, values in $(b,a)$ instantaneously exit from the unstable
region; hence, solutions to \eqref{nonlinpar} are expected to take values generically
in the stable regions.
Thus, a natural point of view is to consider solutions composed of ``pure'' stable phases, 
meaning that, in the decomposition \eqref{plotnikov}, $\lambda^0$ is identically zero and 
$\lambda^\pm$ take values in the binary set $\{0,1\}$.
We refer to such special solutions as {\sf two-phase solutions}.

Assuming the smoothness of the interface separating the two stable phases, Evans and 
Portilheiro\cite{EvaPor04} took advantage of the entropy condition \eqref{entropyPlot} to 
determine pointwise conditions to be satisfied on such transition interface by a piecewise 
smooth solution, in the same fashion as in the case of hyperbolic conservation laws.  
Precisely, given the smooth curve $\gamma:=\big\{(\xi(t),t)\, | \,t \in [0,T]\big\}$, let
\begin{equation*}
 R^\pm:=\big\{(x,t)\,:\, t\in[0,T],\, \pm \big(x-\xi(t)\big)>0\big\},
\end{equation*}
and assume $u$ to be smooth outside $\gamma$  and such that, for any $t\in[0,T]$, 
\begin{equation*}
 u\big(R^-\big)\subseteq (-\infty,b],\qquad 
 u\big(R^+\big)\subseteq [a,\infty),
\end{equation*}
and assume that there exist finite limits
\begin{equation*}
  \lim_{\eta \to 0^+} u(\xi(t) \pm \eta,t)=: u(\xi(t)^\pm,t),\qquad
  \lim_{\eta \to 0^+} \frac{\partial u}{\partial x}(\xi(t) \pm \eta,t)=:
   \frac{\partial u}{\partial x}(\xi(t)^\pm,t).
\end{equation*}
Then condition \eqref{entropyPlot} can be translated in a pointwise condition to 
be satisfied along the phase transition curve $\gamma$.

\begin{proposition}\cite{EvaPor04}\label{prop:entrsol}
  A piecewise smooth weak solution $u$ with the above described
  structure satisfies the entropy condition \eqref{entropyPlot} if and only if:\\
  {\sf i.} the function $u$ is a classical solution of \eqref{nonlinpar}
  outside the curve $\gamma$;\\
  {\sf ii.}  along the curve $\gamma$ there hold
\begin{equation}\label{transmission}
 [\phi(u)]_\gamma=0,\qquad  \xi'(t)\,[u]_\gamma
  +\left[\frac{\partial \phi(u)}{\partial x}\right]_\gamma=0
\end{equation} 
where $[f]_\gamma:=f(\xi(t)^+,t)-f(\xi(t)^-,t)$ denotes the jump
of $f$ at $\gamma$.\\
{\sf iii.} along the curve $\gamma$ the following entropy conditions are satisfied
\begin{equation}\label{entropy} 
\left\{\begin{aligned}
 &\xi'(t)\geq 0  \quad \textrm{if}\; \;  \phi(u)(\xi(t),t)=A \, , \\
 &\xi'(t)\leq 0  \quad \textrm{if}\; \;  \phi(u)(\xi(t),t) = B \, , \\
 &\xi'(t)=0  \quad \textrm{if}\; \;  \phi(u)(\xi(t),t) \in (A,B). 
\end{aligned}\right.
\end{equation}
\end{proposition} 
 
Roughly speaking, a two-phase solution can face alternatively two different
possible evolutions: the steady interface problem and the moving interface one.
\vskip.15cm

{\sf i.} {\sl Steady interface}.
If $\phi(u)(\xi(t),t)\in(A,B)$, then the entropy condition \eqref{entropy} implies that the transition 
interface $\gamma$ is steady.  Moreover, conditions \eqref{transmission} become
\begin{equation*}
 [\phi(u)]_\gamma=0,\qquad
 \left[\frac{\partial \phi(u)}{\partial x}\right]_\gamma=0.
\end{equation*} 
In this case, the solution can be obtained by solving \eqref{nonlinpar}
at the left and at the right of $\gamma$, with the condition of a 
$C^1$-connection for $\phi(u)$ at $\gamma$.
\vskip.15cm

{\sf ii.} {\sl Moving interface}.
If $\phi(u)(\xi(t),t)\in\{A,B\}$ then the entropy condition \eqref{entropy} implies that the 
transition interface $\gamma$ is allowed to move, becoming a free interface of the problem, 
with direction of propagation depending on the value $\phi(u)(\xi(t),t)$.
To fix ideas, let $\phi(u)(\xi(t),t)=A$. 
The values of $u$ at the left-hand and at the right-hand side of $\gamma$ are hence given by
\begin{equation*}
 u(\xi(t)^-,t)=c,\qquad u(\xi(t)^+,t)=a.
\end{equation*}
The second condition in \eqref{transmission}, to be read as a {\it Rankine--Hugoniot condition}, 
is needed to determine the speed of propagation of the free-interface.
\vskip.25cm

A result on local existence and uniqueness of such kind of solutions has been
proved in Ref.~\cite{MaTeTe09}, showing that the approach by Plotnikov is indeed
sufficient to guarantee well-posedness when restricted to solutions with
pure and stable phases.

Going back to the derivation of the complete model equation, it is natural to ask if the 
two-phase description for the limiting solutions to \eqref{relaxing} as $\varepsilon\to 0^+$ 
still holds when considering the limiting solutions to \eqref{hyperbolic} as $\varepsilon,\tau\to 0^+$.  
We are not aware of any result in this direction and we regard the problem of determining an 
appropriate limiting dynamics for \eqref{hyperbolic} as a very challenging direction to investigate.

\subsection*{Riemann problems}
Among other initial-value problems, the Riemann problem for \eqref{nonlinpar},
determined by an initial data of the form
\begin{equation}\label{rpb}
  u^0(x)=\left\{\begin{aligned}
      &u^- \qquad &x<0,\\ &u^+ \qquad &x>0,
  \end{aligned}\right.
\end{equation}
with $u^\pm\in\mathbb{R}$ given, is particularly noteworthy because of its invariance
with respect to the rescaling $(x,t)\mapsto(\lambda x,\lambda^2 t)$.
As a consequence, for any $u^\pm\in\mathbb{R}$, the solution of the Cauchy problem 
\eqref{nonlinpar}--\eqref{rpb} has the form
\begin{equation}\label{autosimilar}
 u(x,t)=f\left(\xi\right) \qquad \textrm{where}\quad\xi:=\frac{x}{\sqrt{t}}
\end{equation}
where the function $f$ is such that $f(\pm\infty)=u^\pm$.  Discontinuity
curves in the $(x,t)-$plane have the form $x=\bar \xi\sqrt{t}$ for some
$\bar\xi\in\mathbb{R}$.  For $\xi\neq\bar\xi$, the function $f$ is solution to the
following (second order) ordinary differential equation
\begin{equation}\label{ssedo} 
  \phi(f)''+\frac{1}{2}\,\xi\,f'=0.
\end{equation}
At any jump point $\bar\xi$, the following relations, obtained from conditions \eqref{transmission},
\begin{equation}\label{rhss}
 \bigl[\phi(f)\bigr]_{\bar\xi}=0,\qquad 
 \bigl[\phi(f)'\bigr]_{\bar\xi}+ \frac{1}{2}\bar\xi\bigl[f\bigr]_{\bar\xi}=0,
\end{equation}
have to hold.  
A detailed analysis of such self-similar solutions to \eqref{nonlinpar}  has been fulfilled
in the case of general initial data $u^\pm$ (hence, also in the unstable region)\cite{GilTes09}.
Here, we focus on the case of a piecewise linear diffusion function $\phi$
\begin{equation}\label{phipiecelin}
 \begin{aligned}
   \phi(u)&=\phi^-(u):=m^-u+q^-\quad \textrm{for}\quad u\leq b,\\
   \phi(u)&=\phi^+(u):=m^+u+q^+\quad \textrm{for}\quad u\geq a,
 \end{aligned}
\end{equation}
for some $m^\pm>0$ and $q^\pm\in\mathbb{R}$.  
In this case, it is possible to assemble an explicit formula for the solution satisfying 
the conditions described in Proposition \ref{prop:entrsol}.  
Such solutions will be useful to test the algorithms proposed in the subsequent Sections.

In order to give an explicit formula for the solution to the Riemann problem, let us introduce 
the functions, related to the fundamental solution of the linear diffusion equation in 
the half-space: given $m>0$, we set
\begin{equation*}
 E_{m}^-(\xi):=\frac{1}{\sqrt{4\,\pi\,m}}\int_{-\infty}^\xi e^{-y^2/4m}\,dy,\qquad
 E_{m}^+(\xi):=\frac{1}{\sqrt{4\,\pi\,m}}\int_{\xi}^{+\infty} e^{-y^2/4m}\,dy.
\end{equation*}
The functions $E^\pm_{m}$ satisfy the properties
\begin{equation*}
	E^\pm_m(\pm\infty)=0,\qquad
	E^\pm_m(\mp\infty)=1\quad\textrm{ and }\quad
	E^-_m+E^+_m\equiv 1.
\end{equation*}  
Without loss of generality, we consider the case $u^-\in(-\infty,b]$ and $u^+\in[a,+\infty)$.

\begin{proposition}\label{prop:rpb}. 
Let $\phi$ be given as in \eqref{phipiecelin}.
The Riemann problem for \eqref{nonlinpar}, determined by the initial datum \eqref{rpb}, 
has a two-phase solution with a steady interface if and only if 
\begin{equation}\label{entropycdngen}
  \frac{\sqrt{m^+}}{\sqrt{m^+}+\sqrt{m^-}}\,\phi(u^-)
  +\frac{\sqrt{m^-}}{\sqrt{m^+}+\sqrt{m^-}}\,\phi(u^+)\in[A,B],
\end{equation}
In this case, the solution has the form \eqref{autosimilar} where 
\begin{equation*}
  f(\xi):=\left\{\begin{aligned}
      &(\phi^-)^{-1}(g(\xi)) \qquad & \xi<0,\\ &(\phi^+)^{-1}(g(\xi))\qquad & \xi>0,
 \end{aligned}\right.
\end{equation*}
and
\begin{equation}\label{exprg}
  g(\xi):=\left\{\begin{aligned}
      &\phi(u^-)E_{m^-}^+(\xi)+\phi(u^+)E_{m^-}^-(\xi)
      -\frac{\sqrt{m^+}-\sqrt{m^-}}{\sqrt{m^+}+\sqrt{m^-}}\bigl[\phi\bigr]\,E_{m^-}^-(\xi)\quad & \xi<0,\\
      &\phi(u^-)\,E_{m^+}^+(\xi)+\phi(u^+)\,E_{m^+}^-(\xi)
      -\frac{ \sqrt{m^+}-\sqrt{m^-}}{ \sqrt{m^+}+\sqrt{m^-}}\bigl[\phi\bigr]E_{m^+}^+(\xi)\quad & \xi>0.
 \end{aligned}\right.
\end{equation}
\end{proposition}

\begin{proof}
Let us look for the self-similar solution $f$ in the form
\begin{equation*}
 f(\xi)=\left\{\begin{aligned}
  &u^-+C^-E_{m^-}^-(\xi) \qquad & \xi<\bar\xi,\\
  &u^+-C^+E_{m^+}^+(\xi) \qquad & \xi\geq \bar\xi,
 \end{aligned}\right.
\end{equation*}
where $\bar\xi, C^\pm$ have to be determined.
In terms of the function
\begin{equation*}
 g(\xi):=\phi(f(\xi))=\left\{\begin{aligned}
  &\phi(u^-)+\kappa^-\,E_{m^-}^-(\xi) \qquad & \xi<\bar\xi,\\
  &\phi(u^+)-\kappa^+\,E_{m^+}^+(\xi) \qquad & \xi\geq \bar\xi,
 \end{aligned}\right.
\end{equation*}
where $\kappa^\pm:=m^\pm\,C^\pm$, the transmission conditions \eqref{rhss} 
give a linear system for the unknowns $\kappa^\pm$
\begin{equation*}
  \left\{\begin{aligned}
      &E_{m^+}^+(\bar\xi)\,\kappa^++E_{m^-}^-(\bar\xi)\,\kappa^-= \bigl[\phi\bigr]\\
      &\left(-\frac{e^{-\bar\xi^2/4m^+}}{\sqrt{m^+}}+ \frac{\sqrt{\pi}\,\bar\xi}{m^+}\,E_{m^+}^+(\bar\xi)\right)\kappa^+
      +\left(\frac{e^{-\bar\xi^2/4m^-}}{\sqrt{m^-}}+\frac{\sqrt{\pi}\,\bar\xi}{m^-}\,E_{m^-}^-(\bar\xi)\,\right)\kappa^-
      =\sqrt{\pi}\,\bar\xi\,\bigl[u\bigr]
 \end{aligned}\right.
\end{equation*}
Setting
\begin{equation*}
 \begin{aligned}
 \Delta(\bar\xi)&:=E_{m^+}^+(\bar\xi)\left(\frac{e^{-\bar\xi^2/4m^-}}{\sqrt{m^-}}
 	+\frac{\sqrt{\pi}\,\bar\xi}{m^-}\,E_{m^-}^-(\bar\xi)\,\right)
  	+E_{m^-}^-(\bar\xi)\left(\frac{e^{-\bar\xi^2/4m^+}}{\sqrt{m^+}}
	-\frac{\sqrt{\pi}\,\bar\xi}{m^+}\,E_{m^+}^+(\bar\xi)\right)\\
  &=\frac{e^{-\frac{1}{4}\,\bar\xi^2(1/m^++1/m^-)}}{\sqrt{m^+\,m^-}}
  	\left\{F_+(\bar\xi)+F_-(\bar\xi)+\frac{\sqrt{\pi}\,(m^+-m^-)}{m^+\,m^-}\,
		\bar\xi\,F_+(\bar\xi)\,F_-(\bar\xi)\right\},
  \end{aligned}
\end{equation*}
where $F_i(\bar\xi):=\sqrt{m^i}\,e^{\bar\xi^2/4m^i}\,E_{m^i}^i(\bar\xi)$, with $i\in\{-,+\}$, 
the values of $\kappa^\pm$ are given by
\begin{equation*}
	\kappa^+=\frac{e^{-\bar\xi^2/4m^-}}{\sqrt{m^-}\,\Delta(\bar\xi)}\,\bigl[\phi\bigr]
	+\frac{\sqrt{\pi}}{m^-}\Bigl(\phi^+(u^+)-\phi^-(u^+)\Bigr)
	\frac{\bar\xi\,E_{m^-}^-(\bar\xi)}{\Delta(\bar\xi)}
\end{equation*}
\begin{equation*}
  \kappa^-=\frac{e^{-\bar\xi^2/4m^+}}{\sqrt{m^+}\,\Delta(\bar\xi)}\bigl[\phi\bigr]
  -\frac{\sqrt{\pi}}{m^+}
  \Bigl(\phi^+(u^-)-\phi^-(u^-)\Bigr)\frac{\bar\xi\, 
    E_{m^+}^+(\bar\xi)}{\Delta(\bar\xi)}. 
      \end{equation*}
In the case $\bar\xi=0$, there hold
\begin{equation*}
 \Delta(\bar\xi)=\frac{\sqrt{m^+}+\sqrt{m^-}}{2\sqrt{m^+\,m^-}},\qquad
 \kappa^\pm=\frac{2\,\sqrt{m^\pm}}{\sqrt{m^+}+\sqrt{m^-}}\,\bigl[\phi\bigr],
\end{equation*}
so that the function $g$ is given by \eqref{exprg}.
Such a formula defines an entropy solution of the problem if and only if
condition {\sf iii.} in Proposition \ref{prop:entrsol} is satisfied, that is,
\begin{equation*}
 \frac{1}{2}\phi(u^-)+\frac{1}{2}\phi(u^+)
      -\frac{1}{2}\frac{\sqrt{m^+}-\sqrt{m^-}}{\sqrt{m^+}+\sqrt{m^-}}\bigl[\phi\bigr]
      \in[A,B],
\end{equation*}
that is equivalent to \eqref{entropycdngen}.
\end{proof}

If \eqref{entropycdngen} does not hold, the Riemann problem exhibits a moving interface.
For the sake of simplicity, let us restrict the attention to the case $m^+=m^-=m$.  
Then the values of $\kappa^\pm$ reduce to
\begin{equation*}
  \kappa^\pm=\bigl[\phi\bigr]\pm\frac{\sqrt{\pi}}{\sqrt{m}} 
  \bigl(q^+-q^-\bigr)\,\bar\xi\,e^{\bar\xi^2/4m}\,E_{m}^\mp(\bar\xi)
\end{equation*}
and the function $g$ can be rewritten as
\begin{equation}\label{decogi}
 \begin{aligned}
   g(\xi)&=\left(\phi(u^-)+\frac{\sqrt{\pi}}{\sqrt{m}}
     \bigl(q^--q^+\bigr)\,\bar\xi\,e^{\bar\xi^2/4m}\,E_{m}^-(\bar\xi)
     \chi_{{}_{[\bar\xi,+\infty)}}(\xi)\right)E_{m}^+(\xi)\\
   &\qquad +\left(\phi(u^+)+\frac{\sqrt{\pi}}{\sqrt{m}}
     \bigl(q^--q^+\bigr)\,\bar\xi\,e^{\bar\xi^2/4m}\,E_{m}^+(\bar\xi)
     \chi_{{}_{(-\infty,\bar\xi)}}(\xi)\right)E_{m}^-(\xi)
 \end{aligned}
 \end{equation}
where $\chi_{{}_{I}}$ denotes the characteristic function of the set $I$.
Condition \eqref{entropycdngen} corresponds to the request 
 \begin{equation*}
   \frac{1}{2}(\phi(u^-)+\phi(u^+))\in[A,B].
 \end{equation*}
If $\frac{1}{2}(\phi(u^-)+\phi(u^+))>B$ (or $<A$), the relation $g(\bar\xi)=B$ ($=A$, respectively) 
is an implicit relation for the value $\bar\xi$ and the transition curve for the corresponding solution 
to the Riemann problem is given by $x=\zeta(t)=\bar\xi\sqrt{t}$.

\section{Semi-discrete approximation of the pseudo-parabolic equation}\label{sec:relaxing}

Next, we aim at designing appropriate numerical schemes for both equation \eqref{nonlinpar},
considered in the setting described in Section \ref{sec:physical}, and the pseudo-parabolic 
equation \eqref{relaxing}.
First of all, we concentrate on the latter, considered in a finite interval $I\subset\mathbb{R}$, with 
homogeneous Neumann boundary conditions,\cite{NovPeg91}
\begin{equation}\label{bdaryNeumann}
 \frac{\partial u}{\partial x}\Bigr|_{x\in\partial I}=0.
\end{equation}
Different types of approximations of pseudo-parabolic equations have been considered in the literature
(finite-element methods Reff.~\cite{ArDoTh81,LiLiRaCa02,TrDu05,GaRu09,GaQiZh09}, spectral methods
Ref.~\cite{Quar87}, finite differences Ref.~\cite{JaJeSa09}\dots).
Here, we consider a standard finite-difference space discretization for equation \eqref{relaxing},
that appears to be continuous at $\varepsilon=0$, so that the limit $\varepsilon\to 0^+$ can be 
analyzed by simply setting $\varepsilon=0$ directly in the algorithm.  
Formally, this amounts to an exchange of order of the two vanishing limits ${h}\to 0^+$ 
and $\varepsilon\to 0^+$.  
Because of the presence of the ratio $\dfrac{\varepsilon}{{h}^2}$, no uniformity of one limit with 
respect to the other parameter into play may hold, in general.  
Hence, there is no guarantee that such an approach leads to the correct limiting solution.  

\subsection*{The semi-discrete scheme}
Let $I:=(0,1)$.
Given $h:=1/(J-1)>0$  for some $J\geq 5$, let us set $x_j:=j\,h$ for $j=1,\dots,J$.
Let $U=U(t)=(U_1(t),\dots,U_J(t))$. 
Let us approximate the second order derivative with respect to $x$ with
\begin{equation*}
 \frac{\partial^2 u}{\partial x^2}(x_j,t)= 
 \frac{u(x_{j+1},t)-2u(x_{j},t)+u(x_{j-1},t)}{h^2}+O(h^2).
\end{equation*}
The boundary conditions \eqref{bdaryNeumann} translate into
the conditions
\begin{equation*}
 U_{1}(t)=U_{2}(t),\qquad U_{J}(t)=U_{J-1}(t).
\end{equation*}
For the sake of simplicity, we denote abusively $(\phi(U_1),\hdots,\phi(U_J))^T$ by $\phi(U)$.  
Equation \eqref{relaxing} can thus be approximated by the ODE system
\begin{equation*}
  \left(h^2\,{\mathbb I}+\varepsilon\,\mathbb{A}\right)
  \frac{dU}{dt}=-\mathbb{A}\,\phi(U),
\end{equation*}
where $\mathbb I$ is the identity matrix and 
\begin{equation}\label{matrixA}
\mathbb{A}:=\left(\begin{array}{ccccccccc}
 1 & -1 & 0 & \cdots& \cdots& \cdots & 0 & 0 & 0\\
 -1 & 2 & -1 & \ddots& \cdots& \cdots & 0 & 0 & 0\\
 0 & -1 & 2 & \ddots& \ddots& \cdots & 0 & 0 & 0\\
 \vdots & \ddots & \ddots & \ddots & \ddots & \ddots & \vdots& \vdots& \vdots\\
\vdots & \vdots & \ddots & \ddots & \ddots & \ddots & \ddots& \vdots& \vdots\\
\vdots & \vdots & \vdots & \ddots & \ddots & \ddots & \ddots& \ddots& \vdots\\
 0 & 0 & 0 & \cdots &\ddots & \ddots & 2 & -1 & 0\\
 0 & 0 & 0 & \cdots &\cdots &\ddots & -1 & 2 & -1\\
 0 & 0 & 0 & \cdots &\cdots &\cdots & 0 & -1 & 1\\
 \end{array}\right).
 \end{equation}
The matrix $\mathbb{A}$ being an irreducible Hessenberg symmetric matrix, $\mathbb{A}$ is
diagonalizable and all its eigenvalues are simple.  
By Gerschg\"orin-Hadamard's circle theorem, the spectrum $\sigma(\mathbb{A})$ of the (symmetric) 
matrix $\mathbb{A}$ is contained in $D(2,2):=\{\lambda\in\mathbb{C}\,:\,|\lambda-2|\leq 2\}$. 
Let us denote the eigenvalues of $\mathbb{A}$ by
\begin{equation}\label{spectrum_of_A}
  \lambda_j:=4\sin^2\left(\dfrac{(j-1)\pi}{2J}\right),\qquad j\in\{1,\hdots,J\},
\end{equation}
as can be inferred from the computation of the eigenvectors. 
As a consequence, the matrix $\mathbb{A}$ is positive semi-definite and the matrix 
$h^2{\mathbb I}+\varepsilon\,\mathbb{A}$ is invertible for any value of positive $\varepsilon$ and $h$. 
Note that $\lambda_{J}$ converges to $4$ as $J\rightarrow\infty$, so that the
upper bound given by Gerschg\"orin-Hadamard's circle theorem is optimal.

The semi-discrete scheme can be rewritten as
\begin{equation}\label{backforSD}
 \frac{dU}{dt}=-\left(h^2{\mathbb I}+\varepsilon\,\mathbb{A}\right)^{-1} \!\! \mathbb{A}\,\phi(U)
 	=:-\mathbb{A}_{h^2,\varepsilon} \phi(U).
\end{equation}
\begin{remark}\label{specdecA}
Since $\mathbb{A}_{h^2,\varepsilon}$ is a polynomial fraction of the matrix $\mathbb{A}$, the spectral theorem 
implies at once that $\mathbb{A}$ and $\mathbb{A}_{h^2,\varepsilon}$ share the same spectral decomposition,
that is their eigenvalues are simple and their eigenvectors are associated respectively with 
$\lambda_j$ and $R(\varepsilon\lambda_j/h^2)/\varepsilon$, for any $j\in\{1,\hdots,J\}$, 
where $R:x\in\mathbb{R}^+\mapsto x/(1+x)$. 
Since for any $x\geq 0$, $0\leq R(x)\leq \min(1,x)$, the spectrum of $\mathbb{A}_{h^2,\varepsilon}$ is 
bounded from above by $\min(1/\varepsilon,4/h^2)$. 
Thus the numerical effect of considering $\mathbb{A}_{h^2,\varepsilon}$ in place of $h^{-2}\,\mathbb{A}$ is 
that the large eigenvalues of $\mathbb{A}$ are filtered.
Note also the difference becomes more apparent when dealing with data in the unstable
region, $\phi'<0$, since for such values the semi-discrete scheme \eqref{backforSD} exhibits 
exponentially increasing behavior at the linearized level.
\end{remark}

Following Remark \ref{specdecA} and the fact that $\mathbb{A}_{h^2,\varepsilon}$ has a regular limit 
as $\varepsilon\to 0^+$, we now focus on the limit of the semi-discretized system
\begin{equation}\label{backforSDlim}
  \dfrac{dU}{d\tau}=-\mathbb{A}\,\phi(U),
\end{equation}
where $\tau=t/h^2$.  
In the original variable $t$, system \eqref{backforSDlim} is stiff as $h\to 0^+$. 
Note that, since $-\mathbb{A}\circ\phi$ is autonomous and globally Lipschitz continuous, the Cauchy
problem for system \eqref{backforSDlim} with an initial condition admits a unique $C^1$ solution 
defined globally. 
The study of the behavior of this semi-discrete solution requires the knowledge of both $L^\infty$ 
and $L^2$ properties of  System \eqref{backforSDlim}.
Since $\phi$ is assumed to be piecewise linear, this system can be seen as
\begin{equation}\label{nonautonomous} 
\dfrac{dU}{d\tau}(\tau)=-M(\tau)U(\tau)+W(\tau),
\end{equation}
where $M(\tau)$ is a $J\times J$ matrix and $W(\tau)$ is a vector of size $J$, $M$ and $W$ 
being  piecewise constant in time: the jumps occur when a point changes phases, that is when 
the interface  moves.

\begin{remark}
From now on, we will concentrate on a particular symmetric piecewise linear $\phi$ given 
in \eqref{pwlinsymm} such that $\phi^\pm(s)=2\,s\mp 3$, $a=-b=1$, $d=-c=2$ and
$B=-A=1$ (see Figure \ref{fig:phi}). 
In this case, condition \eqref{entropycdngen} becomes  $u^-+u^+\in[-1,1]$.
The numerical analysis we will perform makes use of the remarkable spectral properties of 
the matrix $\mathbb{A}$. 
If $\phi$ has different slopes $m^\pm>0$ on the left- and right- hand sides, then one can apply the following analysis 
keeping in mind that $\mathbb{A}$ should be replaced by $\mathbb{A}\mathcal{D}$, where $\mathcal{D}$ is a diagonal 
matrix of the form $\textrm{diag}(m^-,\hdots,m^-,m^+,\hdots,m^+)$, so that $\mathbb{A}$ and $\mathbb{A}\mathcal{D}$ are 
similar matrices and share the same spectral properties.
\end{remark}

\subsection*{Riemann initial data with steady interface}
To analyze the semi-discrete algorithm \eqref{backforSD}, we consider the Riemann initial data
\begin{equation*}
  u^0(x)=\left\{\begin{array}{cc}
      u^-\qquad &x<1/2,\\ u^+\qquad &x>1/2,
 \end{array}\right.
\end{equation*}
where $u^\pm\in\mathcal{S}^\pm$ are chosen such that $u^-\leq b<a\leq u^+$ satisfy \eqref{entropycdngen},
so that the interface should not move and both sides should remain in their respective stable phases.

First of all, let  $L\in\{1,\dots,J\}$ be such that the initial datum $U^0=(U_1^0,\dots, U_J^0)$ 
for the system \eqref{backforSDlim} satisfies
\begin{equation}\label{localization}
U_j^0\leq b<a\leq U_h^0\qquad\qquad 
  \forall\,j=1,\dots,L,\quad h=L+1,\dots,J,
\end{equation}
The behavior of  the solution $U$ can be easily determined since $\phi$ is supposed to be piecewise linear 
and symmetric, as given in \eqref{pwlinsymm}.  
Indeed, the function $V=\phi(U)=2U+(3,\hdots,3,-3,\hdots,-3)^T$ satisfies, at least locally in time, 
the system
\begin{equation}\label{Veq}
 \frac{dV}{ds}=-\mathbb{A}\,V,
\end{equation}
where $s=2\tau=2\,t/{h}^2$, so that, as long as the components of the solution $U$ are
localized as in \eqref{localization}, there holds
\begin{equation*}
  V(s)=\exp(-s\mathbb{A})V(0).
\end{equation*}
Since $\exp(-s\mathbb{A})=\lim\limits_{n\rightarrow+\infty}(\mathbb{I}+s\mathbb{A}/n)^{-n}$, 
$\|(\mathbb{I}+s\mathbb{A}/n)^{-1}\|_\infty= 1$ and $(\mathbb{I}+s\mathbb{A}/n)^{-1}\in(\mathbb{R}^+)^{J\times J}$,
we infer
\begin{equation*}
	\|\exp(-s\mathbb{A})\|_\infty=1\qquad \textrm{and}\qquad \exp(-s\mathbb{A})\in(\mathbb{R}^+)^{J\times J}. 
\end{equation*}
Consequently, if $V_j(0)\in[-1,1]$ for any $j$, then $V_j(s)\in[-1,1]$ for all 
$j\in\{1,\hdots,J\}$ and for any $s>0$. If $V_j(0)\not\in[-1,1]$, that is
$u^+>2$ or $u^-<-2$, and $u^++u^-\in[0,1]$ (resp. $u^++u^-\in[-1,0]$), one can
get back to the previous case by dividing $U^0$ by $u^-$ (resp. $u^+$).
 
Hence, there is no change of phases and $M(\tau)=-\mathbb{A}$ for all $\tau\geq0$. 
This behavior is of course compatible with condition \eqref{entropycdngen}.

Let us now study the asymptotic behavior of the solution $U$. 
The value $0$ belongs to $\sigma(\mathbb{A})$ and $\mathbb{A} r=0$ where
$r_1=(1,\dots,1)^T/\sqrt{J}$.  Therefore, the matrix $P_1:=r_1 r_1^T={\mathbf
  1}/J$, where ${\mathbf 1}$ is the $J\times J$ matrix composed of ones, is
the eigenprojection corresponding to the eigenvalue $0$.
Hence, the matrix $\mathbb{A}$ admits a spectral decomposition of the form
$\mathbb{A}=\dfrac{{\mathbf 1}}{J}+\sum\limits_{j=2}^{J} \lambda_j\,P_j$, where $\lambda_j$ are 
the eigenvalues of $\mathbb{A}$ and $P_j$ are the corresponding eigenprojections.
Therefore
\begin{equation}\label{expSD}
  e^{-\mathbb{A} s}=\frac{{\mathbf 1}}{J}
  +\sum_{j=2}^{J} e^{-\lambda_j\,s}\,P_j
\end{equation}
and a solution to \eqref{Veq} with initial data $V^0$ converges exponentially fast to
\begin{equation*}
 P_1\,V^0=\Bigl(\frac{1}{J}\sum_{j=1}^J\,V_j^0\Bigr)\,(1,\dots,1)^T.
\end{equation*}
This means that the solution $U$ converges to a Riemann-shaped steady state as $\tau\to+\infty$.
The solution to system \eqref{Veq} with initial data $V^0$ is given by
\begin{equation*}
  V(\tau)=e^{-\mathbb{A}\tau}\,V^0=\Bigl(\frac{1}{J}\sum_{j=1}^J\,V_j^0\Bigr)\,(1,\dots,1)^T
  +\sum_{j=2}^{J} e^{-\lambda_j\,\tau}\,P_j\,V^0.
\end{equation*}
If $U^0$ is a Riemann-type initial data with values $(u^-,u^+)$, that is
\begin{equation*}
U_j^0=u^-,\quad \forall\,j=1,\dots,L,\qquad\qquad
U_h^0=u^+\quad \forall\,h=L+1,\dots,J,
\end{equation*}
the jump being at the middle point of the grid, we have
\begin{equation}\label{limit_V}
  V(t/hx^2)=\frac{\phi(u^+)+\phi(u^-)}{2}+\sum_{j=1}^{J} e^{-\lambda_j\,t/h^2}\,P_j\,V^0
\end{equation}
Hence, there holds for all $j\in\{1,\hdots,J\}$
\begin{equation}\label{transvalue}
  \lim_{t\rightarrow+\infty}V_j(t/h^2)= \frac{\phi(u^+)+\phi(u^-)}{2}=v^\infty,
\end{equation}
that is the value towards which $\phi$ converges at the transition interface $\zeta$.  
Thus the solution $U$ converges to a two-phase state
\begin{equation}\label{lim_U}
  U_j^\infty=\left\{\begin{aligned}
    &(\phi^-)^{-1}(v^\infty) & \qquad &\textrm{if}\quad j\leq L,\\
    &(\phi^+)^{-1}(v^\infty) & \qquad &\textrm{if}\quad j\geq L+1.
  \end{aligned}\right.
\end{equation}
Thanks to the spectral properties \eqref{spectrum_of_A} and \eqref{limit_V},
we know that this convergence is exponential of rate
\begin{equation}\label{rate}
	2\lambda_2=\dfrac{8}{h^2}\sin^2\left(\dfrac{\pi}{2J}\right)
	\underset{J\rightarrow+\infty}{\sim}2\pi^2.
\end{equation}

\subsection*{Behavior at a transition point}
Now, let us focus on a model situation in which one point is on the verge of the unstable phase
$\mathcal{U}$, that is there exists $L\in\{2,\hdots,J-1\}$ such that the initial datum is given by
\begin{equation*}
  U^0=(b,\hdots,b,d,d+\delta)^T,
\end{equation*}
where the definitions of $b$, $d$ are described in Figure \ref{fig:phi}, the value $d$ is assumed
at the $(L+1)$-th position  of the vector $V^0$ and $\delta>0$ is a positive perturbation.  
The borderline point $V^0_L$ is attracted to the stable phase $\mathcal{S}^+$ because of the 
jump conditions, but it has to cross the unstable phase $\mathcal{U}$ first.  
Let us determine the expression of $M(\tau)$ at least in a neighborhood of $\tau=0$. 
We already know that the solution $U$ is of class $C^1$.  
At first, note that $U_j(\tau)$ is in the stable phase $\mathcal{S}^+$ for $j\geq L+1$ and that
\begin{equation*}
	\dfrac{dU_j}{d\tau}(0)= \phi(U_{j-1}(0))-2\phi(U_{j}(0))+\phi(U_{j+1}(0))
		=\left\{\begin{aligned} 
	&2\delta	&\qquad &\textrm{ if }j= L+1,\\
	&-2\delta	&\qquad &\textrm{ if }j= L+2,\\
	&0 &\qquad &\textrm{ otherwise},
  \end{aligned}\right.
\end{equation*}
so that 
\begin{equation*}
		U_j(\tau)= \left\{\begin{aligned}
			&-1+o(\tau) 				&\qquad &\textrm{ if }j\leq L,\\
			&2+2\delta\tau+o(\tau)		&\qquad &\textrm{ if }j= L+1,\\
  			&2+\delta- 2\delta\tau+o(\tau) 	&\qquad &\textrm{ if }j= L+2,\\
			&2+\delta+o(\tau) 			&\qquad &\textrm{ if }j\geq L+3.
  		 \end{aligned}\right.
\end{equation*}
Consequently, one has $(dU_{L}/d\tau)(\tau)=4\delta \tau+o(\tau),$ and therefore $U_L$ increases 
and changes phases $\mathcal{S}^-\rightarrow\mathcal{U}$. 
Meanwhile, for $j\leq L-1$, $(dU_{j}/d\tau)(\tau)=o(\tau), $ so that $U_j(\tau)=-1+o(\tau^2)$ and 
$(dU_{L-1}/d\tau)(\tau)=-2\delta \tau^2+o(\tau^2)$.  
By recursion, $U_j$ stays locally in the stable phase $\mathcal{S}^-$, for $j\leq L-1$ and we can 
write the ODE system \eqref{backforSDlim} linearly, at least locally in time in a neighborhood of $t=0$, 
as
\begin{equation}\label{unstable_dyn_sys}
   \frac{dU}{d\tau}=-\mathbb{B}\,U +\mathcal{B},
\end{equation}
with $\mathcal{B}:=(0,\hdots,0,-3,0,3,0,\hdots,0)^T$ and
{
\begin{equation*}
  \mathbb{B}:=2\left(\begin{array}{cccccccccccccc}
      1 & -1& 0 &\hdots &\hdots &\hdots &\hdots &\hdots &\hdots  &\hdots  &0 &0 &0 \\
      -1 & 2 & -1 &\ddots & \hdots & \hdots  &\hdots&\hdots &\hdots  &\hdots  & 0& 0&0\\
      0 & -1 & 2 & \ddots   & \ddots  & \hdots & \hdots&\hdots &\hdots &\hdots   & 0& 0&0\\
      \vdots & \ddots & \ddots & \ddots &\ddots &\ddots  &&& & &\vdots &\vdots &\vdots\\
      \vdots &\vdots & \ddots & -1 & 2 & -1 & \ddots& & & &\vdots &\vdots &\vdots \\
      \vdots & \vdots &\vdots &  \ddots & -1& 2 & 1/2 & \ddots& & &\vdots &\vdots &\vdots  \\
      \vdots & \vdots & \vdots   & & \ddots & -1 & -1& -1 &   \ddots&  &\vdots &\vdots &\vdots  \\
      \vdots & \vdots & \vdots   &   & &  \ddots & 1/2 &  2 & -1&   \ddots& \vdots &\vdots &\vdots\\
      \vdots & \vdots &  \vdots  &   &   & &  \ddots & -1 & 2 & -1&  \ddots&\vdots&\vdots\\
      \vdots & \vdots &  \vdots  &  &    &   & &  \ddots & \ddots  & \ddots  & \ddots &\ddots&\vdots\\
      0& 0 &0 &\hdots &\hdots &\hdots &\hdots&\hdots&\ddots &\ddots   &   2    &  -1     & 0 \\
      0& 0 &0 &\hdots &\hdots &\hdots &\hdots&\hdots&\hdots &\ddots &   -1    &  2     & -1 \\
      0& 0 &0 &\hdots &\hdots &\hdots &\hdots&\hdots&\hdots &\hdots   &  0     & -1     & 1
    \end{array}\right)
\end{equation*}
}where the ``defect'' is located at the $(L-1)$-th, $L$-th and $(L+1)$-th
lines.  The above expression allows us to give a formula for $U(\tau)$, for as
long as $U_L$ stays in the unstable phase $\mathcal{U}$, that is
$\tau\in[0,\tilde{\tau}]$:
\begin{equation}\label{formUunstable}
  U(\tau)=\exp(-\tau\mathbb{B})U(0)+\left(\int_0^t\exp(-s\mathbb{B})ds\right)\mathcal{B}.
\end{equation}
In order to fully exploit \eqref{formUunstable} and predict the time $\tilde{\tau}$ at which 
$U_L$ shall leave the unstable phase $\mathcal{U}$, let us describe the spectrum of $\mathbb{B}$.

\begin{proposition}\label{prop:specB}
The matrix $\mathbb{B}$ is diagonalizable in $\mathbb{R}$ with $J$ distinct eigenvalues which can be ordered 
as $-1<\mu_{-1}<\mu_0=0<\mu_1<\hdots<\mu_{J-2}<8$. 
In particular, $\mathbb{B}$ has only one negative eigenvalue.
\end{proposition}

\begin{remark}
The negative eigenvalue of $\mathbb{B}$ is the one that allows the borderline point to go through the 
unstable phase $\mathcal{U}$.
\end{remark}

\begin{proof}
The matrix $\mathbb{B}^T$ can be written as $\mathbb{B}^T=D\mathbb{A}$ where $D=\textrm{diag}(2,\hdots,2,-1,2,\hdots,2)$.  
Let us prove that $\mathbb{B}^T$ is diagonalizable.  
As mentioned before, $\mathbb{A}$ is symmetric non-negative and can be diagonalized in an orthonormal 
basis $(r_0,\hdots,r_{J-1})$. 
Let $\Lambda:=\textrm{diag}(0,\lambda_1,\hdots,\lambda_{J-1})$, 
$\widehat\Lambda:=\textrm{diag}(1,\lambda_1,\hdots,\lambda_{J-1})$, which is invertible, 
and $W$ be the unitary matrix the columns of which are $(r_0,\hdots,r_{J-1})$ in that order, so that
\begin{equation*}
  W=\left(\begin{array}{ccc}
      \dfrac{1}{\sqrt{J}}& & w^T\\&&\\\dfrac{1_{J-1,1}}{\sqrt{J}} && \widetilde{W}
    \end{array}\right),
\end{equation*}
with $w\in\mathbb{R}^{J-1}$ and $\mathbb{A}=W \Lambda W^T$ where we denote by $\widetilde{M}$
the $(1,1)$-submatrix of size $(J-1)\times(J-1)$ of a $J\times J$ matrix $M$.
Let us note that
\begin{equation}\label{W}
  \dfrac{1}{J}1_{J-1,J-1}+\widetilde{W}\,\widetilde{W}^T
  	=ww^T+\widetilde W^T\,\,\widetilde W=\mathrm{Id}_{J-1,J-1},
\end{equation}
so that, in particular, $\widetilde{W}$ is invertible and $\widetilde W^{-1}=\dfrac{J}{J-1}\widetilde W^T$.  
The matrix $C:=W^T \mathbb{B}^T W = W^T D W W^T \mathbb{A} W =W^T D W \Lambda$ is similar to $\mathbb{B}^T$. 

Let us now multiply on the left- and right-hand sides of $C$ by
$\widehat\Lambda^{1/2}$ and $\widehat\Lambda^{-1/2}$:
\begin{equation*}
  \widehat\Lambda^{1/2}C\widehat\Lambda^{-1/2} = 
  	\widehat\Lambda^{1/2}W^T D W \Lambda^{1/2}
  	= \left(\begin{array}{cc} 0 & c^T \\
    		0_{J-1,1} & \widetilde\Lambda^{1/2}\Delta\widetilde\Lambda^{1/2}
 			\end{array}\right)
\end{equation*}
where $c\in\mathbb{R}^{J-1}$ and $\Delta:=\widetilde W^T\widetilde{D}\widetilde W+2ww^T$ which is symmetric. 
Let us show now that $\Delta$ is invertible and has exactly one negative and $J-2$ positive eigenvalues, 
so that $\mathbb{B}^T$ (and $\mathbb{B}$) is similar to a diagonalizable matrix with one negative, one zero and $J-2$
positive eigenvalues using Sylvester's law of inertia.

Using relation \eqref{W}, one can rewrite $\Delta$ as
\begin{equation*}
  \Delta=2 Id + \widetilde W^T(\widetilde{D}-2  Id)\widetilde W.
\end{equation*}
The definition of $D$ implies that 
\begin{equation*}
	\widetilde{D}-2 Id=-3 E_{LL}
\end{equation*}
where $E_{L-1L-1}:=(\delta_{i=L-1,j=L-1})_{1\leq i,j\leq J-1}$, $\delta$ being the Kronecker symbol. 
The spectrum of $\Delta$ is thus $2$ with multiplicity $J-2$ and associated eigenspace 
$\widetilde W^{-1}(\mbox{Span}(e_j)_{j\neq L-1})$ and $-\dfrac{J-3}{J}$ associated with 
$\widetilde W^{-1} e_{L-1}$. 
This proves that $\mathbb{B}$ is diagonalizable with simple eigenvalues with the announced signs.

Applying Gerschg\"orin-Hadamard's circle theorem to $\mathbb{B}^T$, we have
\begin{equation*}
  \mathrm{Sp}(\mathbb{B})\subset D(4,4)\cup D(-2,2),  
\end{equation*}
so that we get the upper bound of the spectrum quite easily. 
To prove the lower bound, we need to examine more closely the computation of the
eigenvalues: assume $(\mu,V)$ is an eigenvalue/eigenvector of $\mathbb{B}^T$. 
Solving $\mathbb{B}^TV=\mu V$ leads to solving two linear recurrences of order 2 that have the
same characteristic polynomial with compatibility conditions. 
The computation gives the following result: if $\mu\in[-4,0)\cup(0,8)$ is an eigenvalue of $\mathbb{B}^T$, 
then
\begin{equation*}
 \begin{aligned}
  \chi(\mu) &:=-2(\mu+2)\left[v_+^{J-2}(\mu)+v_-^{J-2}(\mu)\right]
  	+(4+\mu(2-\mu))\left[v_+^{J-1}(\mu)+v_-^{J-1}(\mu)\right]\\ 
   	&\hspace{1cm} +\,3\,\mu\,\left[v_+^{J-2L-1}(\mu)+v_-^{J-2L-1}(\mu)\right]\\
  	& =  0
\end{aligned}
\end{equation*}
where $v_{\pm}(\mu)=1-\dfrac{\mu}{4}\pm\sqrt{\dfrac{\mu}{2}\left(\dfrac{\mu}{8}-1\right)}$.
Note that, for all $k\in\mathbb{N}$, 
\begin{equation*}
  v_+^k(\mu)+v_-^k(\mu)=\sum_{p=0}^{[(k-1)/2]}\left( \begin{array}{c} k\\2p \end{array}\right)
  	\left(1-\dfrac{\mu}{4}\right)^{k-2p}\left(\dfrac{-\mu}{2}\right)^{p} \left(1-\dfrac{\mu}{8}\right)^{p}
\end{equation*} 
is a $\mu$-polynomial of degree $k$, so that $\chi$ is also a $\mu$-polynomial, of degree $J+1$. 
The possible eigenvalue $\mu=8$ is not valid in $\chi$ since it corresponds to a double root in $v_\pm$. 
However, $\chi(8)=0$, but a further computation shows that $8$ is not an eigenvalue of $\mathbb{B}$: it implies 
that the characteristic polynomial of $\mathbb{B}^T$ is $\tilde\chi:=\chi/(8-\mu)$. 
We already know that there is only one negative eigenvalue of $\tilde\chi$. 
Proving that $\tilde\chi(-1)\tilde\chi(-\infty)>0$ will show that this negative eigenvalue is necessarily larger 
than $-1$.  Since $v_\pm(-1)=2^{\pm 1}$, we get $\tilde\chi(-1)=-(2^{1-J}+2^{J-2L-1}+2^{2L+1-J})/3<0$. 
Thanks to the expression of $v_+^{J-1}(\mu)+v_-^{J-1}(\mu)$, one sees that it is positive as $\mu$
tends to $-\infty$ so that $\tilde\chi(-\infty)<0$. 
This ends the proof of Proposition \ref{prop:specB}.
\end{proof}

Since $\mathbb{B}$ is diagonalizable with simple eigenvalues, one can compute explicitly the 
spectral projectors: let $\mu$ be an eigenvalue of $\mathbb{B}$ and $(v,w)\in(\mathbb{R}^J\setminus\{0\})^2$ 
such that $\mathbb{B} v=\mu v$, $\mathbb{B}^Tw=\mu w$. 
Since $w^Tv\neq0$, one can assume $w^Tv=1$, so that the spectral projector associated to 
$\mu$ reads $P_\mu=vw^T/(w^Tv)$. 
The direct consequence is the following expression for $U(\tau)$, $\tau\in[0,\tilde{\tau}]$, 
from \eqref{formUunstable}:
\begin{equation}\label{U-proj-gen}
  U(\tau)=P_0(U(0)+\tau\mathcal{B})
  	+\sum_{\mu\in\mbox{\small Sp}(\mathbb{B})\setminus\{0\}}P_\mu\left\{e^{-\tau\mu}U(0)
	+\dfrac{1-e^{-\tau\mu}}{\mu}\mathcal{B}\right\}
\end{equation}
so that, identifying the asymptotically larger terms and noting that, since
$P_0=J\, D^{-1}e\, e^T/(J-1)$,  there holds $P_0\mathcal{B}=0$, we infer
\begin{equation}\label{U-proj}
  U(\tau)=e^{-\tau\mu_{-1}}P_{\mu_{-1}}\left(U(0)-\dfrac{1}{\mu_{-1}}\mathcal{B}\right)
  	+P_0U(0)+\sum_{\mu\in\mbox{\small Sp}(\mathbb{B})\setminus\{0\}}
	\dfrac{1}{\mu}P_{\mu}\mathcal{B}+o(e^{-\mu_1\tau}).
\end{equation}
An approximation of the exit time $\tilde{\tau}$ is thus given by
\begin{equation}\label{exit-time}
  \tilde{\tau}\simeq \dfrac{\log\left(a-[P_0U(0)](L)-\sum_{\mu\in\mbox{Sp}(\mathbb{B})\setminus\{0\}}
	\dfrac{[P_{\mu}\mathcal{B}](L)}{\mu}\right)}
	{\log\left(\left[P_{\mu_{-1}}\left(U(0)-\dfrac{\mathcal{B}}{\mu_{-1}}\right)\right](L)\right)}.
\end{equation}
Note that formulas \eqref{formUunstable}-\eqref{U-proj-gen}-\eqref{U-proj}-\eqref{exit-time}
hold for general piecewise-linear $\phi$, changing $\mathbb{B}$ to a similar matrix and modifying $\mathcal{B}$.

In conclusion, within the time interval $[0,\tilde{\tau}]$, there is only one component, namely $U_L$,
that lies in the unstable phase $\mathcal{U}$ and, as soon as it has entered the stable phase $\mathcal{S}^+$, 
the matrix of the dynamical system is again as in \eqref{Veq}. 
It means that the interface between $\mathcal{S}^-$ and $\mathcal{S}^+$ has moved to the left, as predicted 
by \eqref{entropycdngen}. 
One is then faced to two independent constant coefficient heat equations, that can be treated as in the steady 
configuration case. 
If, on the right-side, the perturbation $\delta$ is too large, $U_{L+1}$ will at some point be larger than $B$ 
and the situation will revert to that with a point in the unstable phase $\mathcal{U}$. 
The system will then move the interface to the right-hand side again by making this point go through the
unstable phase $\mathcal{U}$, until the border is reached or the discretization is such that $\phi(U_{L+1})$ 
is less that $B$.

In the general case of initial data of Riemann type, satisfying \eqref{localization}, and such that $|u^-+u^+|>1$,  
the semi-discrete differential system will make $U_L$ converge exponentially fast to the ``closest'' border 
predicted by the entropy conditions \eqref{entropy}:  
if $u^-+u^+<-1$, then  $(U_L(t),U_{L+1}(t))$ goes to $(c,a)$ and otherwise $(U_L(t),U_{L+1}(t)$ goes to $(b,d)$. 
The rate of convergence to this intermediate state is again 
$2\lambda_2=8\sin^2(\pi/(2J)/h^2\underset{J\rightarrow+\infty}{\sim}2\pi^2$. 
Once the border is reached, we are faced with a situation where the dynamical system is the same 
as in \eqref{unstable_dyn_sys}. 
The analysis previously developed still holds: the borderline point ($U_L$ if $u^-+u^+>1$ or 
$U_{L+1}$ otherwise) is the only one going through the unstable phase $\mathcal{U}$. 
Once it has reached the opposite stable phase, the following point ($U_{L-1}$ or $U_{L+2}$), 
that has been monotonically changing (decreasing or increasing) goes to the border, then 
enters $\mathcal{U}$ and so on, so that the interface moves as predicted by Proposition
\ref{prop:entrsol} in the continuous setting.

\section{Numerical experiments in the limiting regime $\varepsilon=0$}\label{sec:numerelaxing}

Now that we have shown that the semi-discrete scheme plays the role of simulating either 
the convergence to a Riemann-type steady state or the movement of the interface, we turn our
attention to performing concrete numerical experiments, by adding an appropriate discretization
for the time derivative.

There are two ways of fully discretizing the problem, remembering that equation \eqref{nonlinpar} 
is fully non-linear and $\phi$ is not differentiable, so that Newton's algorithm cannot be used in 
a fully implicit scheme:
\vskip.15cm

-- either discretize explicitly in time and space, that is introducing a time-discretization 
$t^n=n\Delta t$ with step $\Delta t$ and using an explicit Euler scheme
\begin{equation}\label{backforD}
  U^{n+1}=U^n-\frac{\Delta t}{{h}^2}\,\mathbb{A}\,\phi(U^n),
\end{equation}
where $(U_j^n)_{j\in\{1,\hdots,J\},n\in\mathbb{N}}$;\\
-- or use the previous analysis and discretize \eqref{nonautonomous} with an implicit 
scheme, globally in time if the initial data is of steady Riemann type or else by estimating the 
exit/entrance time
\begin{equation}\label{backforDBE}
	U^{n+1}=U^n-\dfrac{\Delta t}{{h}^2}\left(\,M(t^{n+1})\,U^{n+1}-W(t^{n+1})\right).
\end{equation}
\vskip.15cm

Using \eqref{backforD} allows to obtain a numerical solution without having to consider exit/entrance 
times, but the CFL condition for stability is $4\Delta t/{h}^2\leq 1$ and the order of approximation 
is $1$ in time, whereas \eqref{backforDBE} avoids $L^2$ and $L^\infty$ stability conditions and
increase the order of the approximation but requires the knowledge of $M(t)$.
Here, we implement both approaches and compare the results.

First of all, we analyze the case of initial data in stable phases, in order to validate the schemes 
by comparing to the exact solution for the Riemann problem stated in Proposition \ref{prop:rpb}, 
with the left-hand state $u^-$ belonging to $\mathcal{S}^-$ and
the right-hand state $u^+$ to $\mathcal{S}^+$, with $\mathcal{S}^\pm$  defined in \eqref{phases}.
Specifically, we choose the case $(u^-,u^+)=(-1,1)$, for which the predicted behavior is 
that of a steady interface and continuity and differentiability of $\phi(u)$ at $x=0.5$. 
\begin{figure}[ht]
  \centering
  \subfigure[$u$ as a function of
  $x$]{\includegraphics[width=7cm,keepaspectratio]{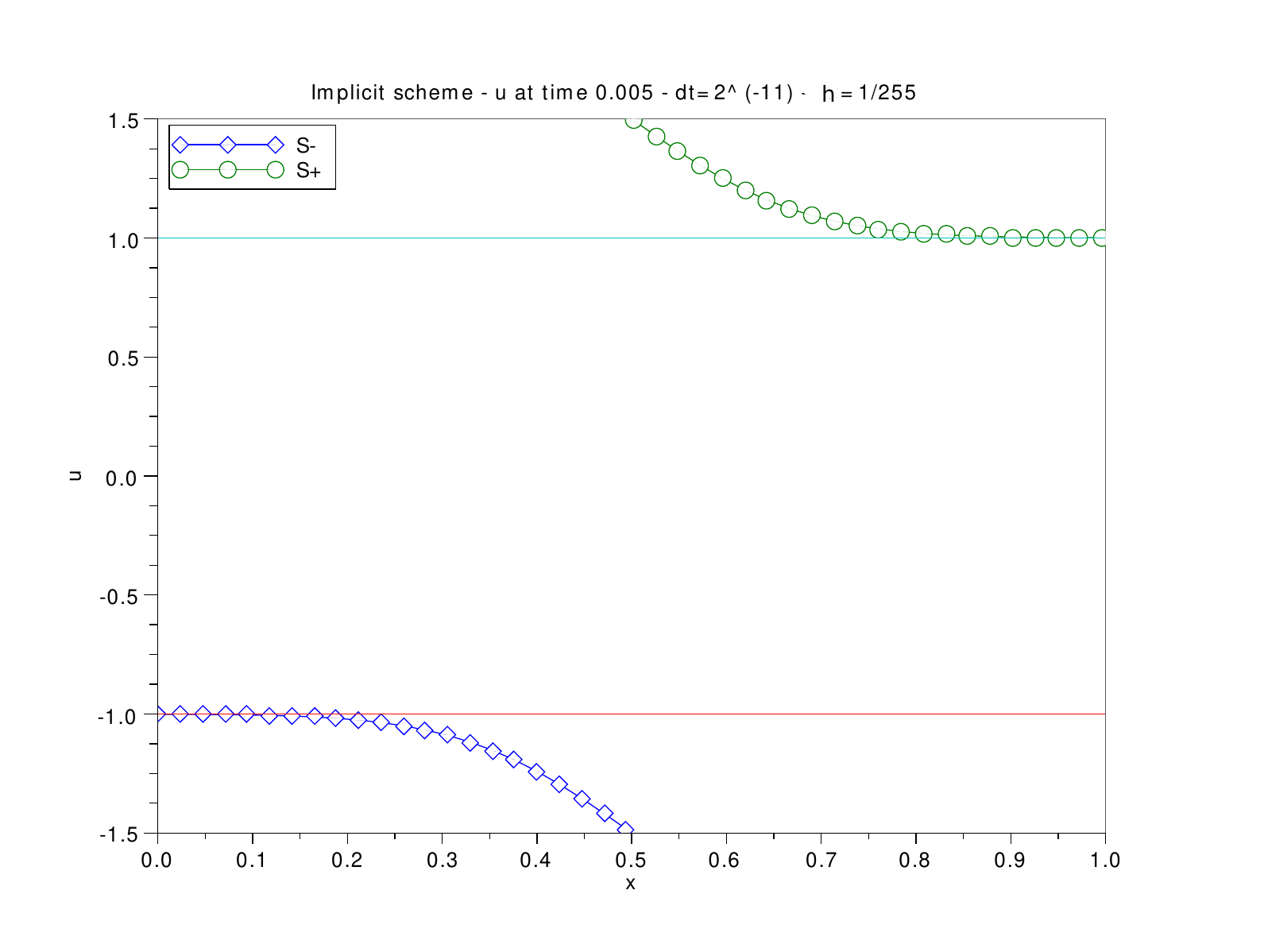}}
  \subfigure[$\phi(u)$ as a function of
  $u$]{\includegraphics[width=7cm,keepaspectratio]{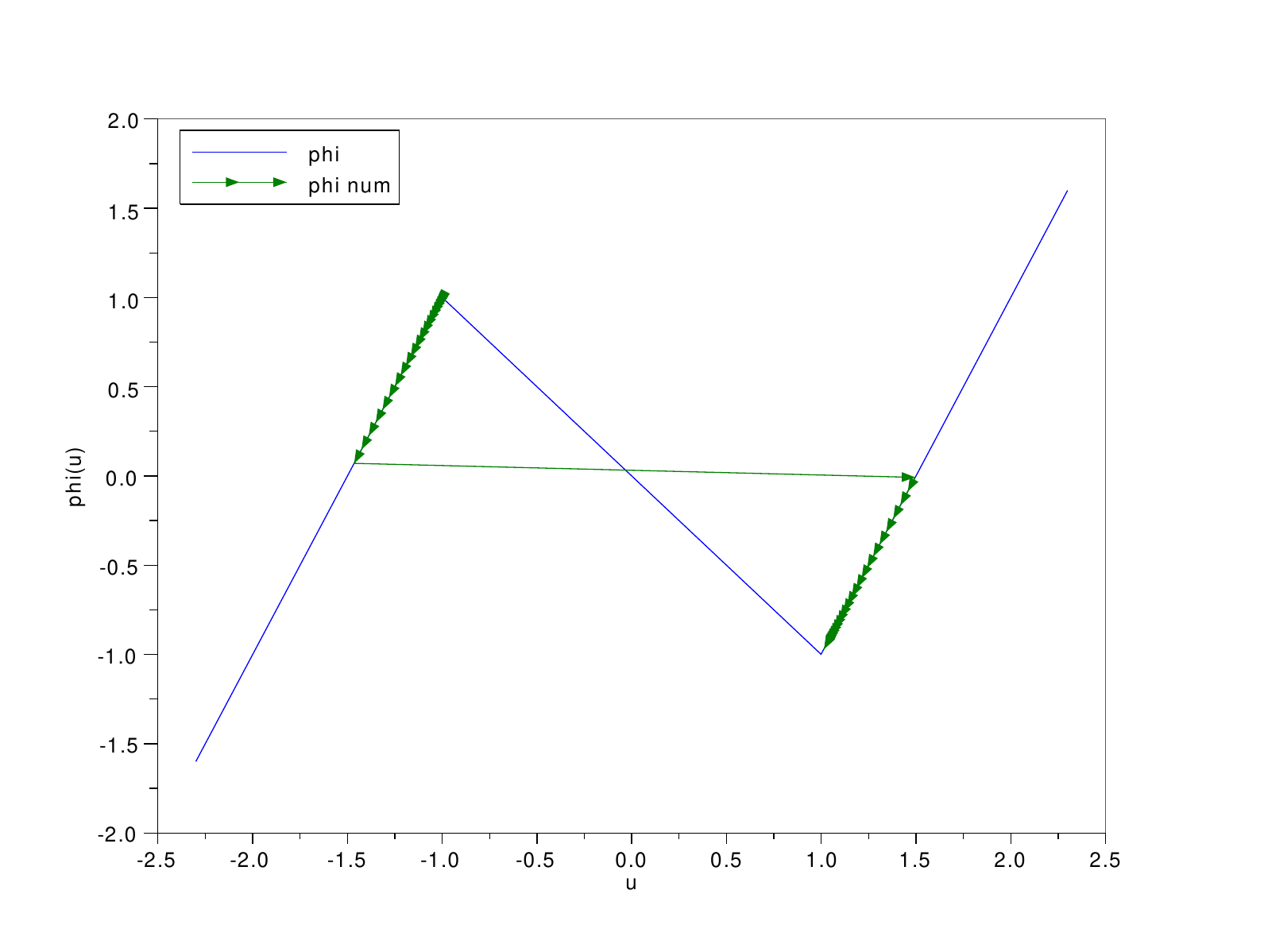}}
  \caption{\small Numerical solution at time $T=0.05$ with Riemann initial data $(-1,1)$.} 
  \label{fig:exsol}
\end{figure}   
The convergence error to the exact solution given by \eqref{decogi} can be expressed as
\begin{equation}\label{cv_err}
  E_2:=\sqrt{{h}\sum_{j=1}^J(U_j^n-u(t^n,x_j))^2}=O(\Delta t^\alpha)+O({h}^\beta).
\end{equation}
In the case of standard diffusion equations with regular initial data, e.g. taking the exact solution 
at small $t$ for $(u^-,u^+)$ in the same stable phase, under the required CFL condition 
$\Delta t=O({h}^2)$, the order of the explicit scheme is $\alpha=1$, $\beta=2$. 
The implicit scheme requires no stability condition and is also of order $\alpha=1$, $\beta=2$, 
the counterpart being that we should solve a system at each step. 
Here, since the initial data are step functions, one cannot hope for such good estimates. 

Numerically, we show here that the order of approximation is $\beta=1$ in ${h}$ for both 
schemes and between $0.5$ and $1$ in $\Delta t$ for the implicit scheme. 
The higher-order region in $\Delta t$ could be explained by the fact that, ${h}$ being fixed, 
as we refine $\Delta t$, the solution is smoothed out faster. 
In conclusion, since in the steady case the system to be inverted is linear, tridiagonal, the
implicit scheme should be preferred (see Table 1). 
\begin{center} {\footnotesize\begin{table}[h]
      \begin{tabular}{|c|c|c|}\hline
        implicit (var $\Delta t$) &implicit (var ${h}$) & Explicit ($\Delta t={h}^2/4)$\\\hline
        \begin{tabular}{c|c|c|c}
           $K$ & $J$  &
          error &cpu\\\hline
          \hline
12	&		&	0.0043659	&	0.02	\\
14	&		&	0.0034607	&	0.06	\\
18	&		&	0.0033399	&	0.62	\\\hline
20	&	9	&	0.0033395	&	3.05	\\\hline
22	&		&	0.0033394	&	9.98	\\
&&\\
&&\\
&&\\
&&\\
        \end{tabular}
        &
        \begin{tabular}{c|c|c|c}
           $K$ & $J$  &
          error & cpu\\\hline
          \hline
&	6.	&	    0.0264534  	&	   1.76  	\\
&	7.	&	    0.0132975  	&	    1.94   	\\
&	8.	&	    0.0066676  	&	    2.31  	\\\hline
20&	9.	&	    0.0033395  	&	    3.05    	\\\hline
&	10.	&	    0.0016730  	&	    4.48    	\\
&11.	&	    0.0008409  	&	    7.71  	\\
&	12.	&	    0.0004287  	&	    14.37 	\\
&	13.	&	    0.0002299  	&	    33.90  	\\
&&\\
        \end{tabular}&
 \begin{tabular}{c|c|c}
    $J$  &
   error & cpu\\\hline
   \hline
6&	0.0373874	&	0.02	\\
7&	0.0188008	&	0.13\\
8&	0.0094283	&	0.16	\\\hline
9&	0.0047216	&	0.80	\\\hline
10&	0.0023638	&	3.87	\\
11&	0.0011850	&	25.54	\\
12&	0.0005981	&	193.12	\\
13&	0.0003099	&	1971.36	\\
14&	0.0001749	&	17407.04	\\
\end{tabular}
\\
\hline
\end{tabular}\label{tab:comp_stable}
\vskip.25cm
\caption{\small Comparison of $L^2$ errors of $\phi(u)$ and execution times (cpu) for
  the implicit and Explicit schemes  wrt $\Delta
  t=2^{-K}$, ${h}=2^{-J}$.}
\end{table}}
\end{center}

Next, we consider a Riemann-type initial data with values $(u^-,u^+)=(-2,3)$.  
Since $\phi(-2)=-1$ and $\phi(3)=3$, the function $\phi(u)$ is discontinuous at 
time $t=0$ with a jump located at the middle point $x=0.5$.   
In short time, as expected, the function $\phi(u)$ is regularized and, at $x=0.5$,
takes a value in the interval $[A,B]=[-1,1]$ depending on the choice of $u^-, u^+$
(see Figure \ref{fig:prof_pwls_zero}).
As a consequence of such a regularization of $\phi$, the profile of $u$
is forced to bend at the left-hand and right-hand side of the phase transition
point.  From now on, diffusion in the two stable phases will flatten the
profile of $u$ on the two phases and, at the same time, the interaction
between the two stable phases will determine the position of the transition
interface and of the corresponding value of $\phi(u)$ on the interface itself.
\begin{figure}[ht] 
  \centering
  \subfigure[$u$ as a function of $x$]{\includegraphics[width=7cm,keepaspectratio]{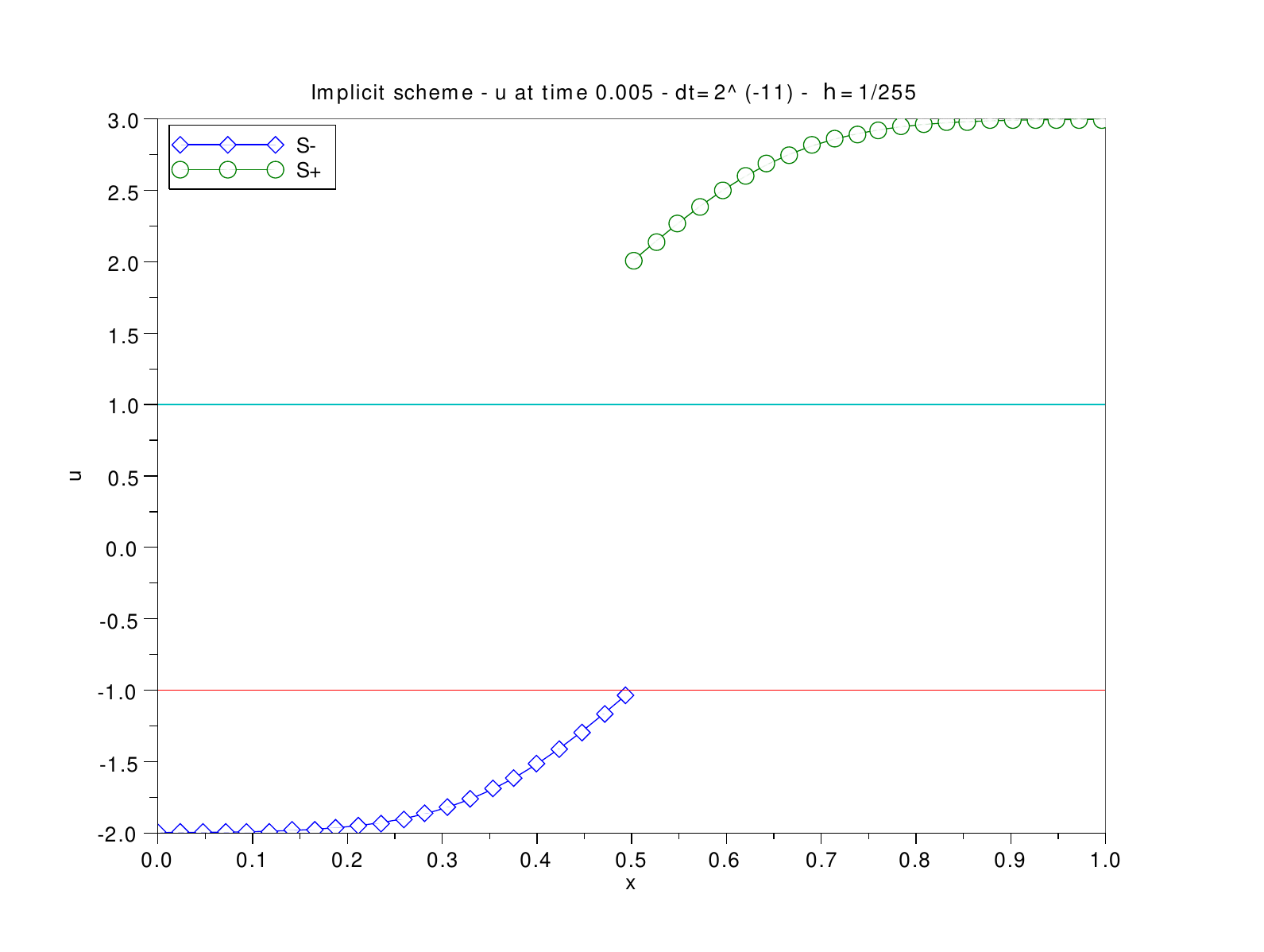}}
  \subfigure[$\phi(u)$ as a function of $u$]{\includegraphics[width=7cm,keepaspectratio]{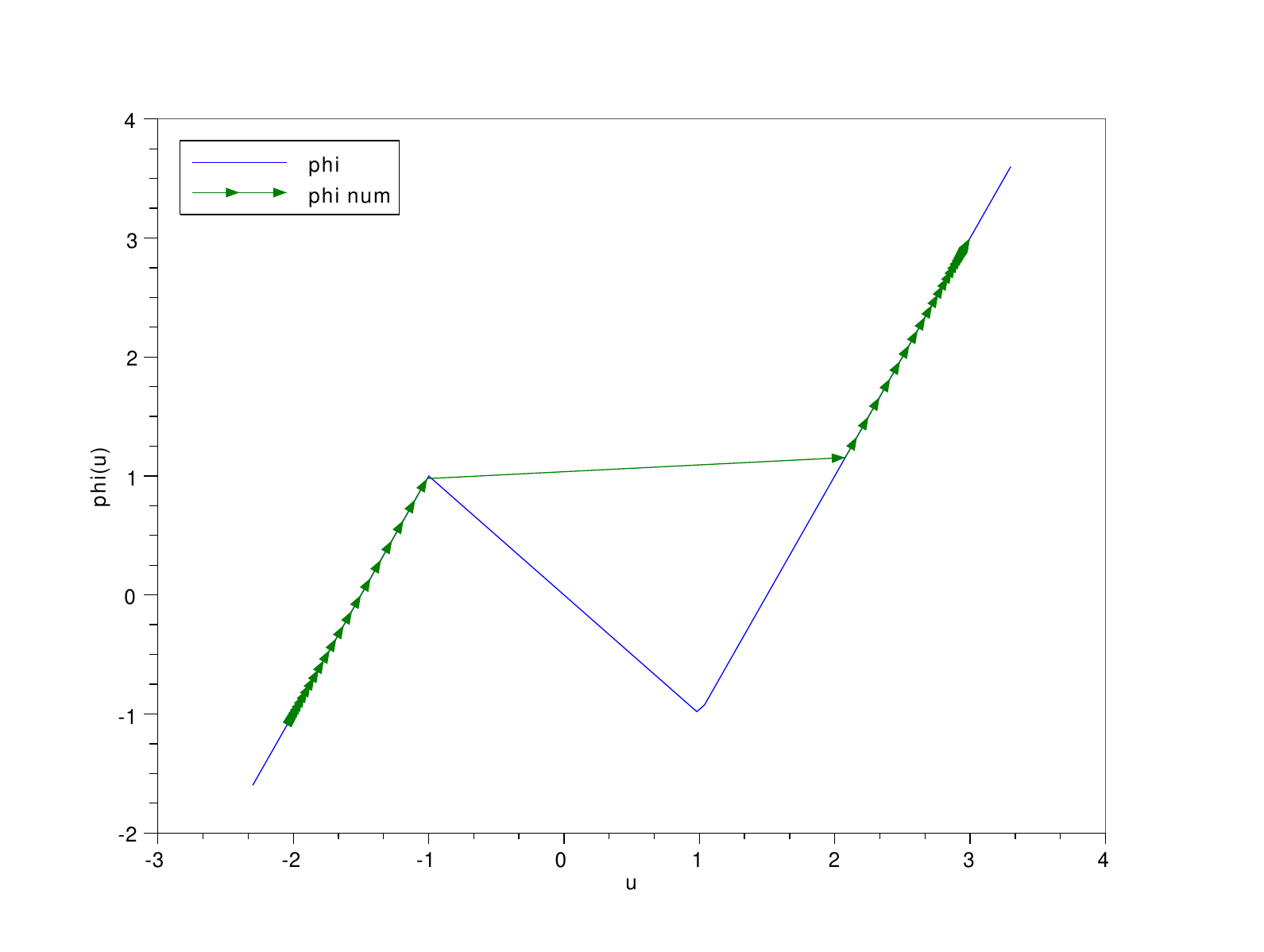}}
   \caption{\small\footnotesize  Short time evolution of the Riemann data $(-2,3)$.
    The two graphs show the profile of $u$ and $\phi(u)$ at the final time $t=5\times 10^{-3}$.}
\label{fig:prof_pwls_zero}  
\end{figure}

For large time, the solution $u$ converges to a Riemann--shaped steady state
and the profile of $\phi(u)$ tends to a constant exponentially fast, as shown in Figure
\ref{fig:profiles_pwls_-2_3}. 
As the continuous theory predicted \eqref{entropycdngen}, since
$\phi(u^-)+\phi(u^+)=2\in[2A,2B]$, the interface did not move. The same
asymptotic behavior is shown by different choices of initial values 
$(u^-,u^+)$.  As long as we were dealing with steady cases, we were able to use the
implicit scheme without further study.
\begin{figure}[ht] 
  \centering
  \subfigure[$u$ as a function of $x$]{\includegraphics[width=6cm,keepaspectratio]{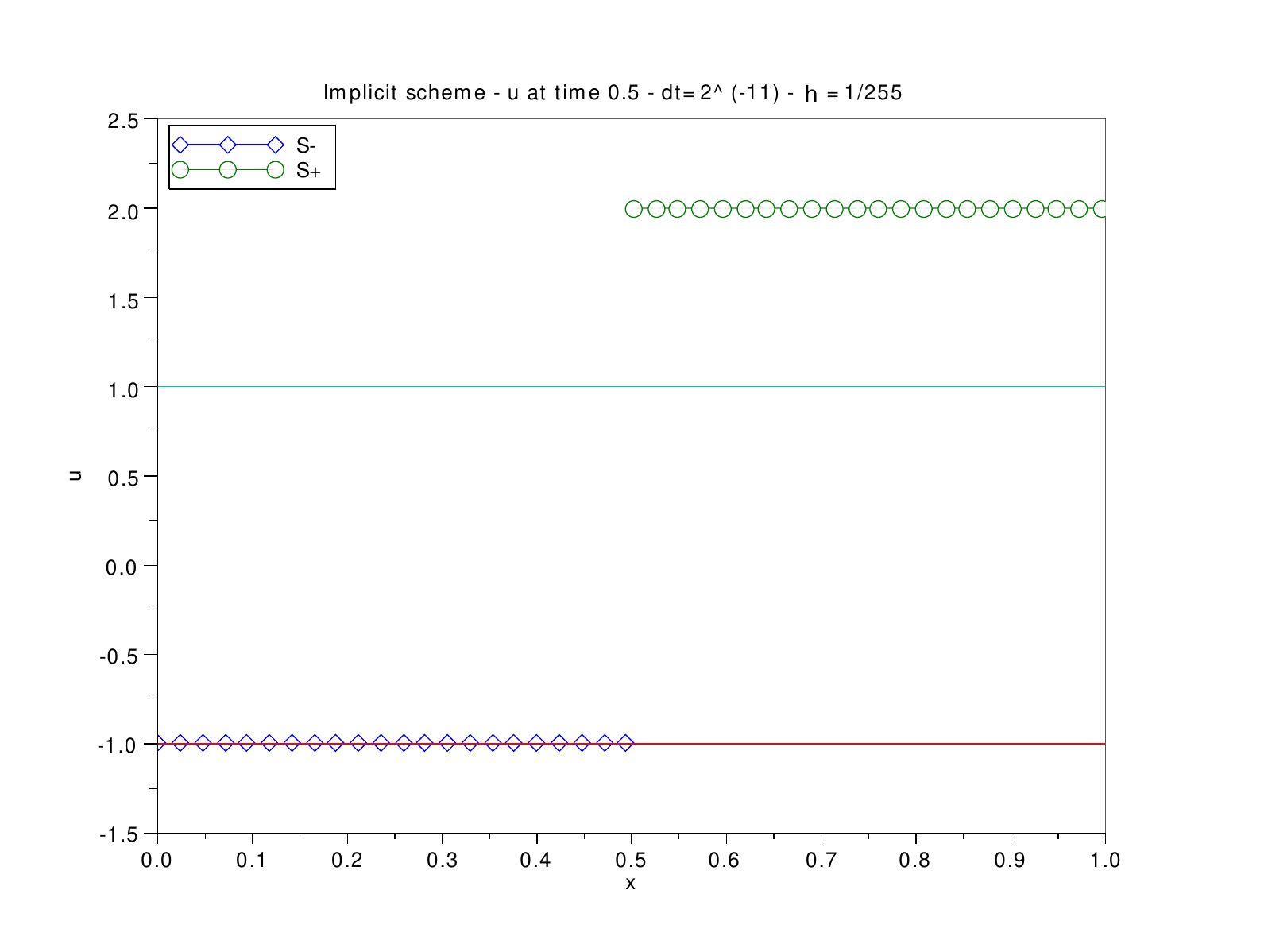}}
  \subfigure[error as a function of $x$]{\includegraphics[width=6cm,keepaspectratio]{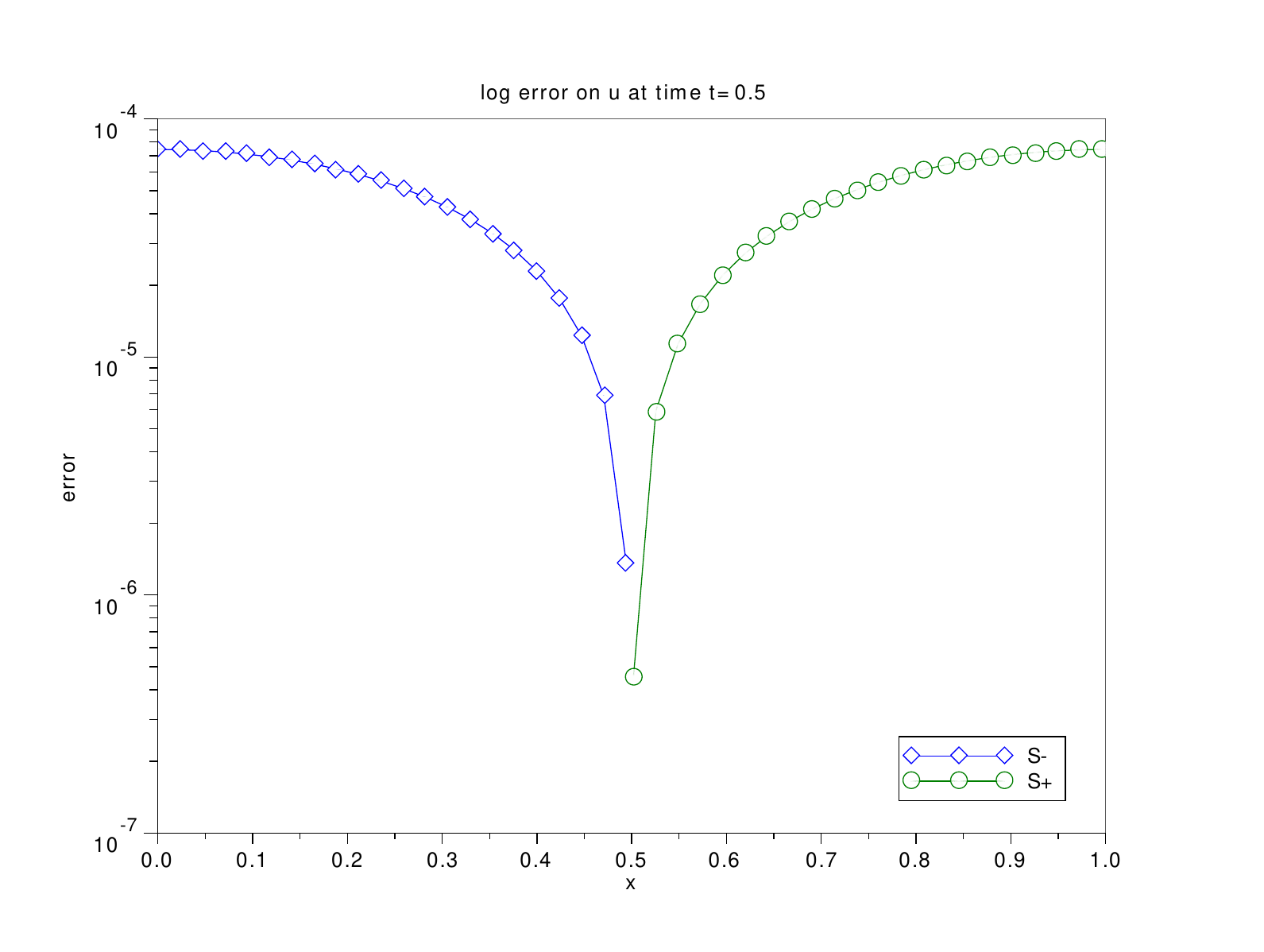}}
  \caption{\small\footnotesize  Large time evolution of the Riemann data $(-2,3)$.
   Letting the time go as far as $t=0.5$ ensures an error of order $10^{-4}$ according to \eqref{rate}.}
\label{fig:profiles_pwls_-2_3}
\end{figure}

A higher level in the jump between the initial states $u^-$ and $u^+$ leads to a
movement of the transition interface separating the stable phases, as shown in 
Figure \ref{fig:profiles_pwls_-2_4}, corresponding to the initial data $(-2,4)$.
We used an explicit scheme to treat this case since a point might have to go 
through the unstable phase. 

\begin{figure}[ht]
  \centering
  \subfigure[$u$ as a function of $x$]
  	{\includegraphics[width=6cm,keepaspectratio]{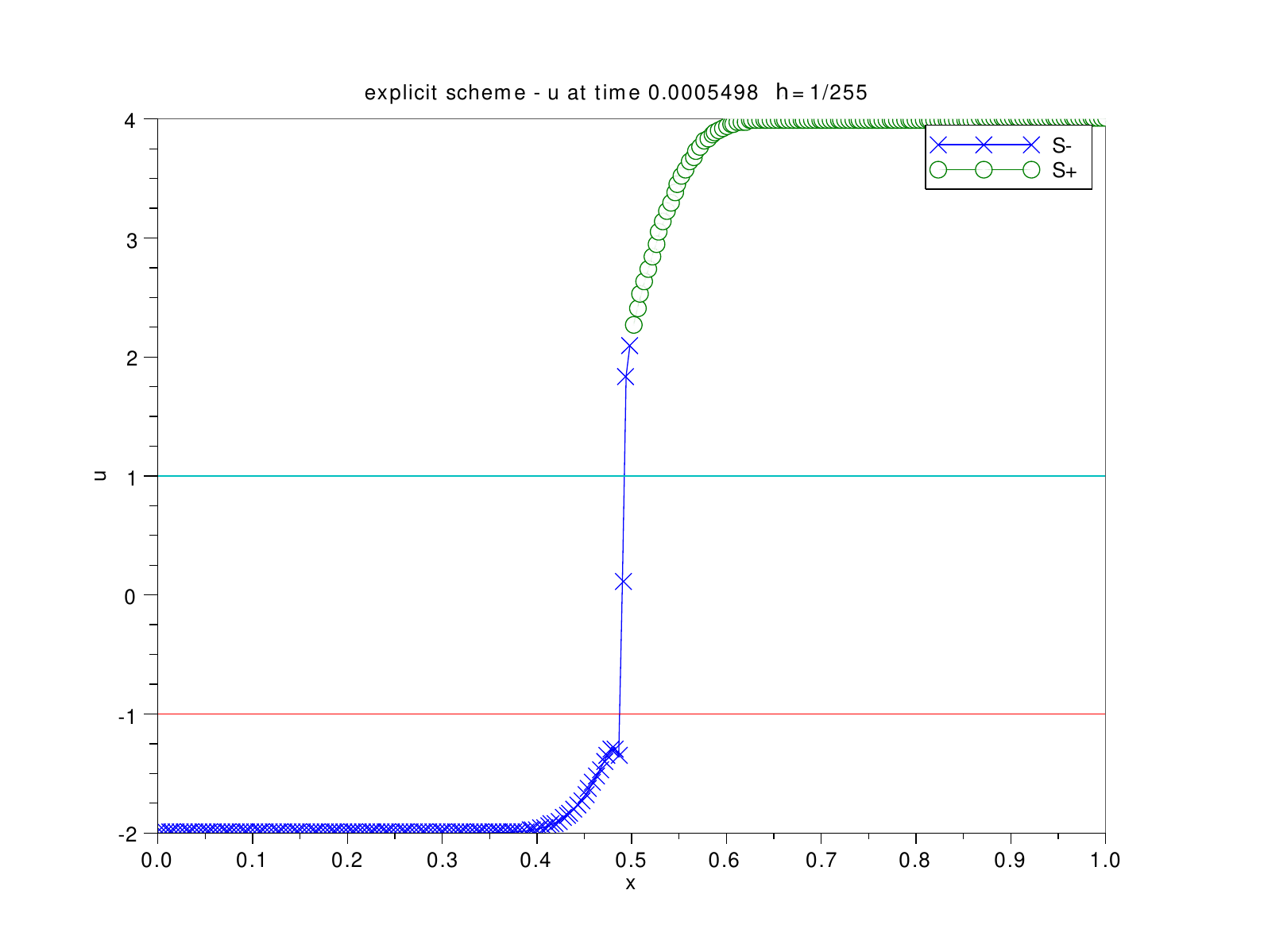}}
  \subfigure[$\phi(u)$ as a function of $u$]
  	{\includegraphics[width=6cm,keepaspectratio]{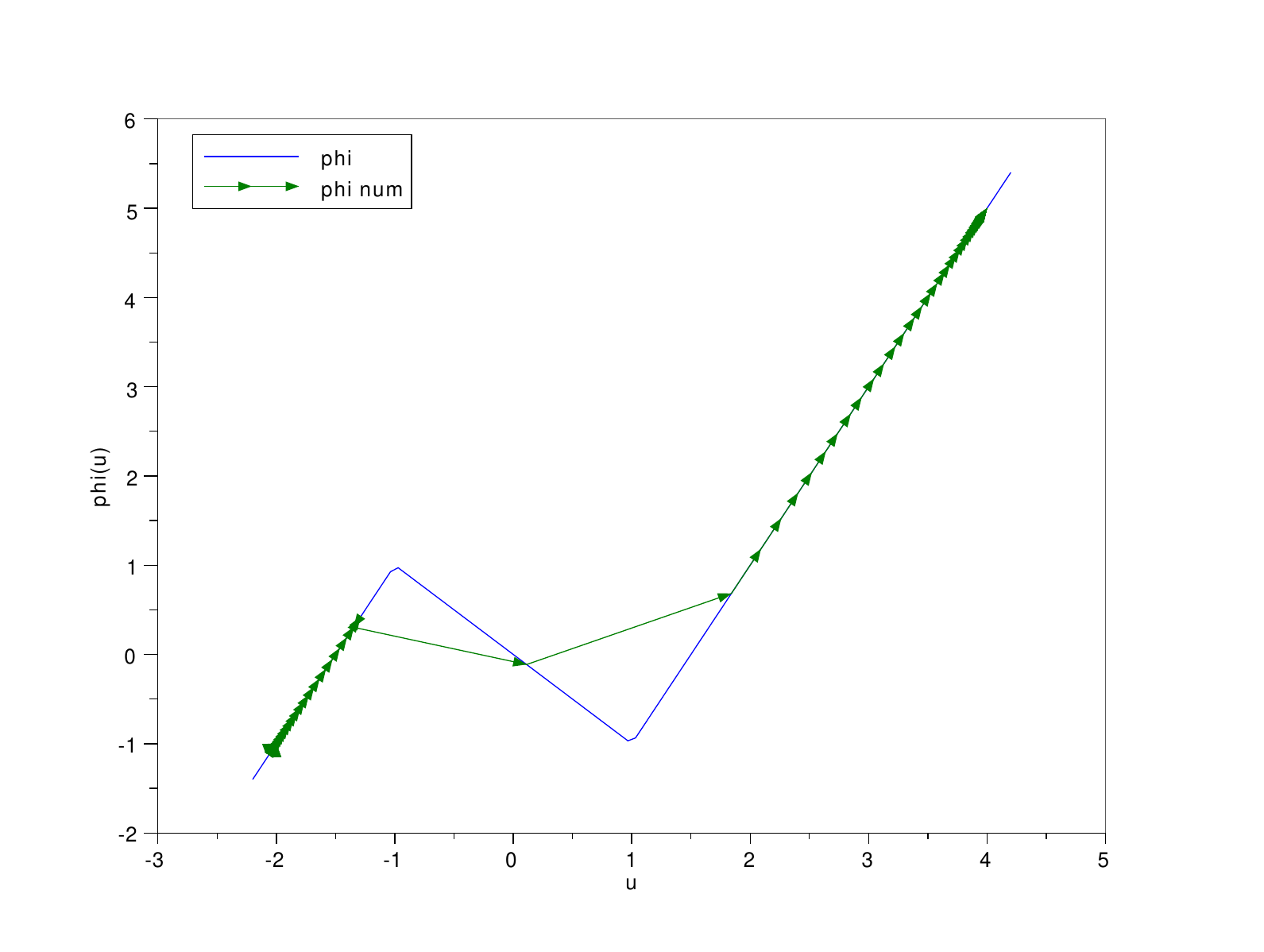}}\\
  \subfigure[$\phi(u)$ as a function of $x$]
  	{\includegraphics[width=6cm,keepaspectratio]{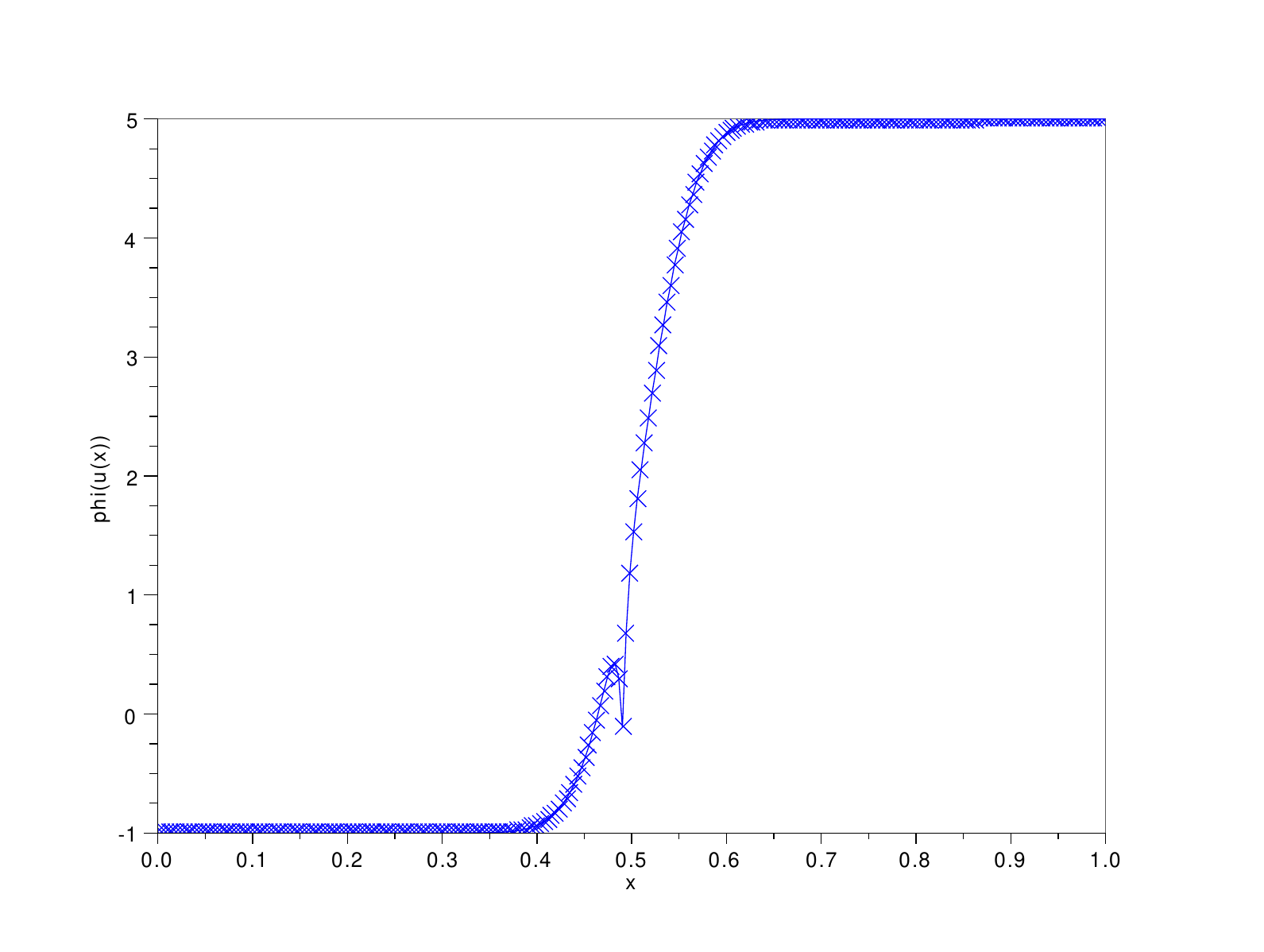}}
\subfigure[zoom of $\phi(u)$ around the interface]
	{\includegraphics[width=6cm,keepaspectratio]{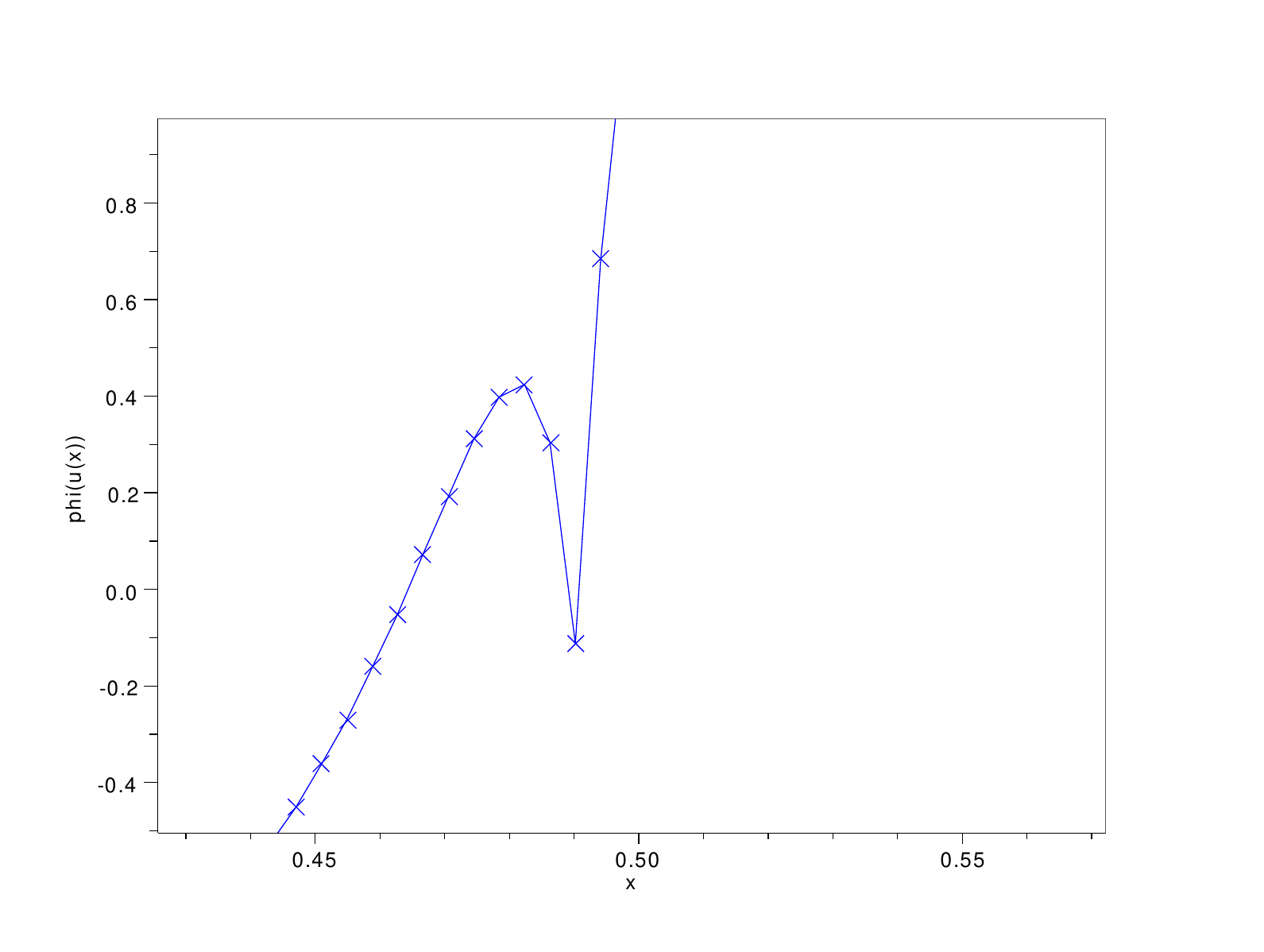}}
\caption{\small\footnotesize Small time evolution of the Riemann data $(-2,4)$.  
On the right-hand side, the graph in the phase--space $(u,\phi(u))$ with corresponding values 
of the solution, where a point in the spinodal region is clearly seen.}
\label{fig:profiles_pwls_-2_4}
\end{figure}

The analysis of the movement of the transition interface $\zeta$ starting at
$x=0.5$ at initial time $t=0$ is very delicate.  First of all, we have to
define the position of the interface.  Here, we consider the interface $\zeta$ in
the semi-discretized scheme to be determined by finding the maximal (resp. minimal) value of
$j$ such that $U_j\leq -1$ (resp. $U_j\geq 1$). 
Here, the position $\zeta=\zeta(t)$ of the interface is given by $\zeta(t)=k(t)\cdot{h}$, 
$k(t)$ being the index localizing the left-hand side of the interface. 
The explicit scheme allows us to get an approximation of the movement of the interface (see 
Figure \ref{fig:mvt_explicit}), from which one recognizes the parabolic shape predicted 
by \eqref{autosimilar}.
However, the use of this scheme yields oscillations.
\begin{figure}[ht]
  \centering
   \subfigure[Movement of the interface]{\includegraphics[width=7cm,height=5cm]{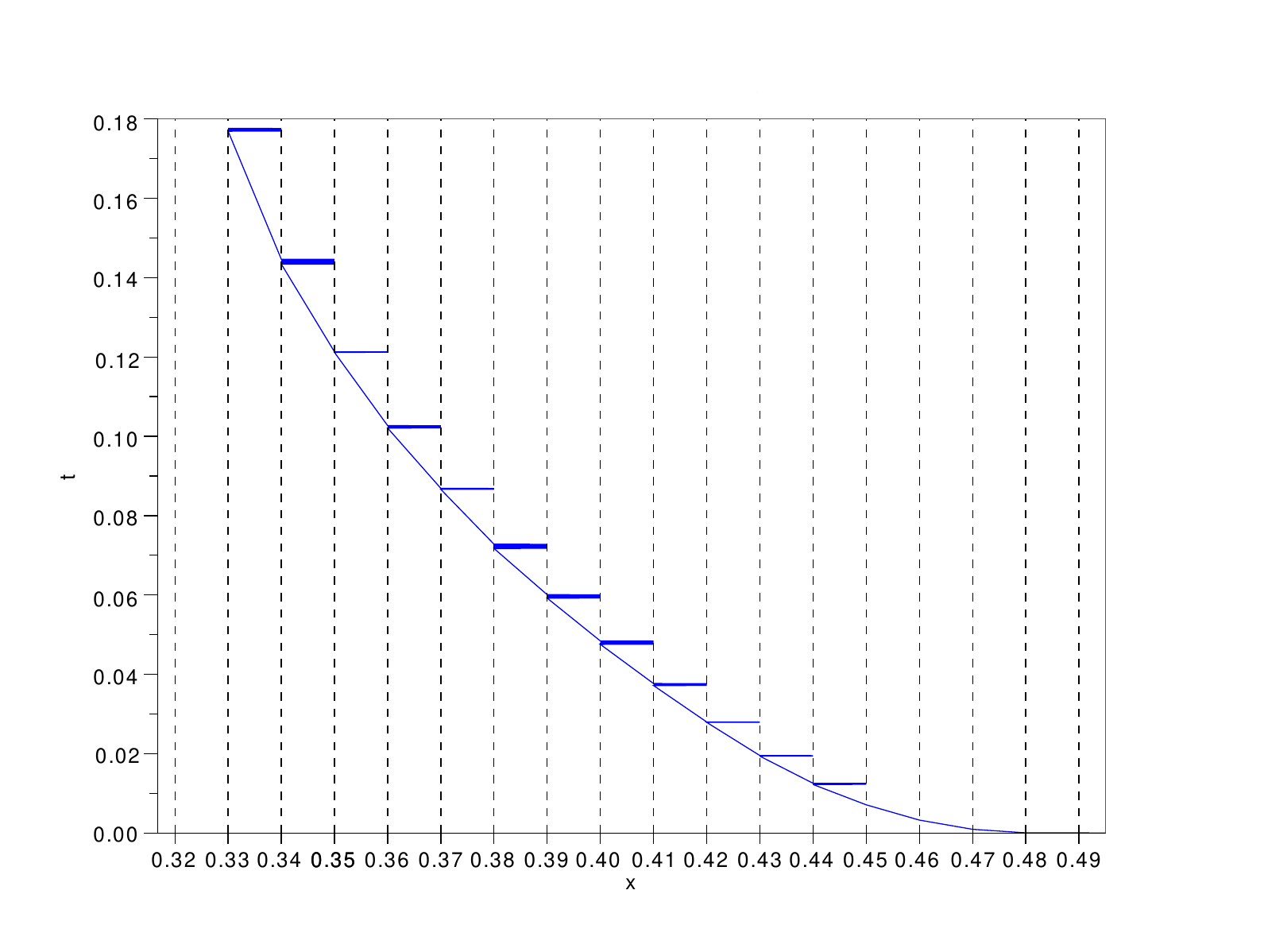}}
 \subfigure[Values of {$-[\partial_x(\phi(u))]/[u]$}]{\includegraphics[width=7cm, height=5cm]{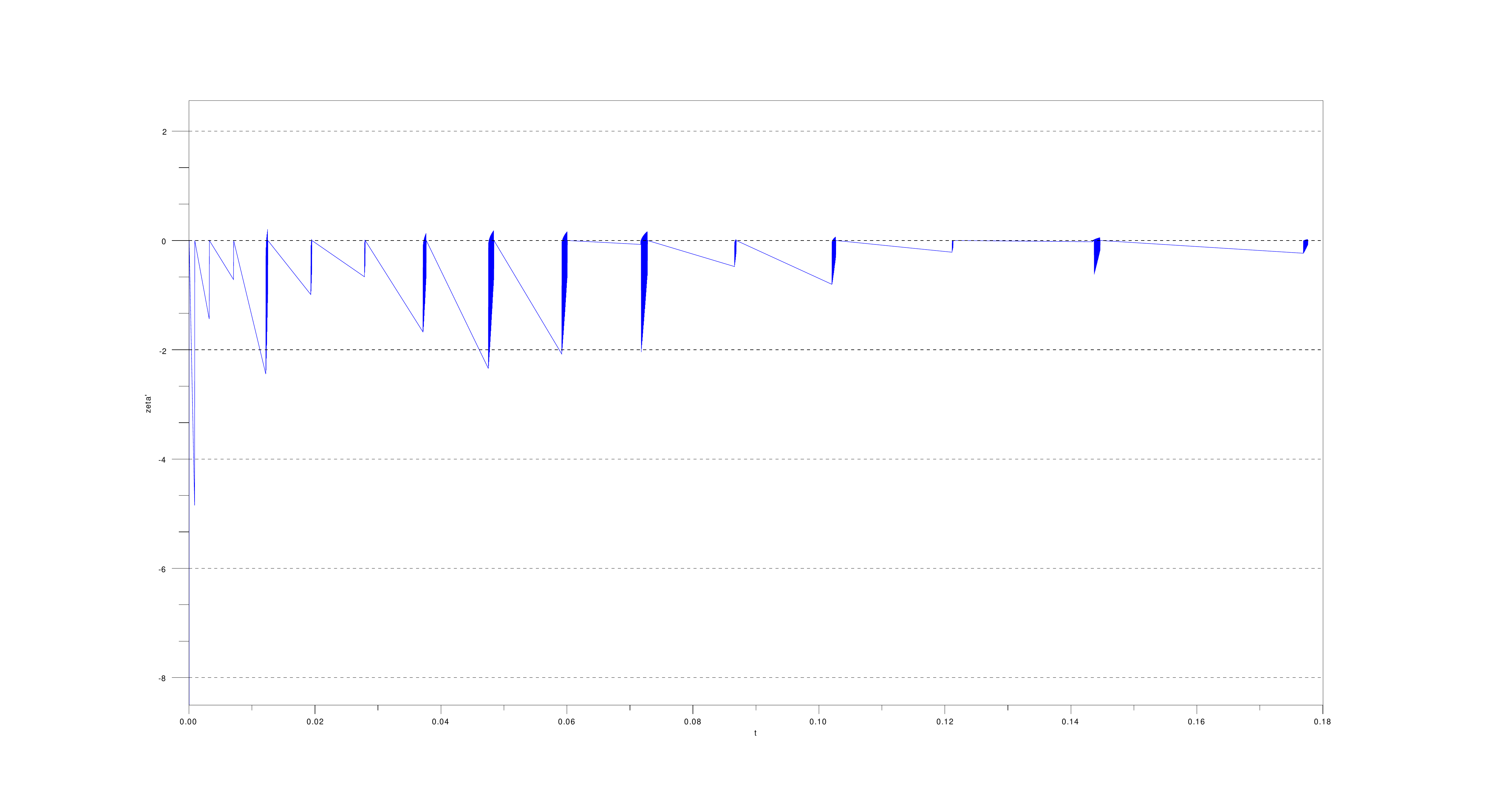}}
  \caption{\small Riemann initial data $u^-=-2$, $u^+=4$ with ${h}=0.01$.}
  \label{fig:mvt_explicit}
\end{figure}

The space derivative $\phi(u)_x$ of the function $\phi(u)$ varies very rapidly close
to the interface $\zeta$, consistently with the presence of a movement of the
interface.  Nevertheless, the algorithm finds as value $\phi(u)\bigr|_{\zeta}$
an uncorrected level (in the limit $\varepsilon\to 0^+$, left--moving
interface are allowed if and only if $\phi(u)=B=1$).  In fact, in the
continuous limiting model the unstable phase $(a,b)$ is prohibited and the
solution $u$ jumps instantaneously from one phase to the other. On the
contrary, in the semi-discretized algorithm \eqref{backforSDlim}, the
transition from one phase to the other is accomplished by smoothly passing
from values under the level $a=-1$ to values above the level $b=1$, thus
passing all way through the spinodal region $(a,b)$ (see Figure \ref{fig:profiles_pwls_-2_4}(c)-(d)). 
Correspondingly, the function $\phi$ decreases its value from $B$ to $A$ and then increases again to $B$. 
Such a transition is faster as ${h}\to 0^+$, but it is always (incorrectly) present.

\section{Two--phase scheme}\label{sec:twophase}

In this Section, we consider the problem of determining a numerical scheme for
the limiting equation \eqref{nonlinpar} together with the entropy conditions described 
in Proposition \ref{prop:entrsol}.  
The algorithm is based on the idea that initial data taking values only in the stable regions 
give rise to solutions with values in the stable regions matched in order to satisfy the 
appropriate transmission conditions.
In the interior of any region, where the solution is in one of the stable phases, a diffusion 
equation has to be solved; and the procedure has to be completed with an appropriate 
description of the transition interfaces behavior, based on the entropy conditions.

At each step, additionally to the values $U^n\in\mathbb{R}^J$ giving the approximate
value of the solution, we also track the non-integer position $\zeta^n$ of the transition
interface determining the passage from one phase to the other.  
The index of such position is denoted by $j_*^n$.

\subsection*{Description of the algorithm}
Let $\phi^\pm\,:\,\mathbb{R}\to\mathbb{R}$ be two strictly increasing extensions of restrictions 
$\phi\bigr|_{(-\infty,b)}$ and $\phi\bigr|_{(a,+\infty)}$, respectively (where, for the 
sake of simplicity,  $\phi$ is piecewise linear and symmetric).
For $j\in\{2,\dots,J-1\}$, let us introduce the projection matrix $\Pi_j$
\begin{equation*}
 \Pi_j:=\left(\begin{array}{ccc} \mathbb{O}_{j-1} &  0 & 0\\
 0 & \mathbb{I}_2 & 0 \\ 0 & 0  & \mathbb{O}_{J-j-1} \end{array}\right),
 \end{equation*}
 where $\mathbb{O}_j$ and $\mathbb{I}_j$ denotes the null and the identity $j\times j$
 matrices, respectively.  
 Finally, given $A<B$, let us introduce the {\sf truncation function} $T$ defined by
\begin{equation}\label{truncation}
 T(U):=\max\{A,\min\{U,B\}\}.
\end{equation}
 The algorithm reads as follows.  Let $U\in\mathbb{R}^J$ and $j_*\in\{1,\dots,J\}$ be
the values for the approximate solution and approximate interface location.  Let
$C:=C(U,j_*)$ be the transition value suggested by \eqref{transvalue}, i.e. set
\begin{equation*}
 C=C(U,j_*):=\frac{1}{2}\left(\phi_-(U_{j_*-1})+\phi^+(U_{j_*+2})\right).
\end{equation*} 
If $C\notin[A,B]$, that is the transition value is not in the hysteresis
interval, we use the truncation function \eqref{truncation} to transform it in
an admissible value.  Specifically, if $C>B$, then $T(C)=B$; if $C<A$, then
$T(C)=A$.  Correspondingly, the values of $U$ at $j_*$ and $j_*+1$ are mapped
in the values $(\phi^-)^{-1}(T(C))$ and $(\phi^+)^{-1}(T(C))$, respectively (i.e.
$b$ and $d$ if $C>B$, and $c$ and $a$ if $C<A$).  If $C\in[A,B]$, the
transition value is in the hysteresis loop and the new values for $U$ at $j_*$
and $j_*+1$ have to be determined, imposing the transmission conditions
\begin{equation*}
  \bigl[\phi(u)\bigr]= \bigl[\phi(u)_x\bigr]=0,
\end{equation*}
that is, in discretized form,
\begin{equation*}
 \phi(U_{j_*})=\phi(U_{j_*+1}),\qquad
 \phi(U_{j_*})-\phi(U_{j_*-1})=\phi(U_{j_*+2})-\phi(U_{j_*+1}).
\end{equation*}
These conditions give
\begin{equation}\label{trans2}
 \phi(U_{j_*})=\phi(U_{j_*+1})
  =\frac{1}{2}\left(\phi(U_{j_*-1})+\phi(U_{j_*+2})\right)=C(U,j_*).
\end{equation}
Hence, let us set
\begin{equation}\label{defF}
 \begin{array}{ll}
 \mathbb{F}(U,j_*):=
  (0,\dots,0&,\underbrace{(\phi^-)^{-1}(T(C))},(\phi^+)^{-1}(T(C)),0,\dots,0)\\
 &\quad\textrm{\tiny $j$-th element}
 \end{array}
\end{equation}
Then, setting $\tilde U^{n}:=\mathbb{F}(U^n,j^n_*)+(\mathbb{I}-\Pi^{j^n_*})\,U^n$,
the scheme can be written as 
\begin{equation}\label{schemerelaxed}
 \left\{\begin{aligned}
  &\Pi_{j^n_*}U^{n+1}=\mathbb{F}(U^n,j^n_*),\\
  &(\mathbb{I}-\Pi_{j^n_*})U^{n+1}=
	 (\mathbb{I}-\Pi_{j^n_*})\left(U^n-\frac{\Delta t}{{h}^2}
	\,\mathbb{A}\,\phi(\tilde U^{n})\right)
	\end{aligned}\right.
\end{equation}
where $\mathbb{A}$ is the matrix introduced \eqref{matrixA}.  Note that the scheme
\eqref{schemerelaxed} coincides with \eqref{backforD}, $\varepsilon=0$, apart
from the definition at the transition points $j_*$ and $j_*+1$.

To define the transition interface sequence, given $U\in\mathbb{R}^J$ and
$j_*\in\{1,\dots,J\}$, let us define the values for the (discrete)
time-derivative of the interface $\zeta$ as follows
\begin{equation*}
 \zeta'(C)=-\frac{1}{(\phi^+)^{-1}(T(C))-(\phi^-)^{-1}(T(C))}
  \left(\frac{\phi(U_{j_*+2})-T(C)}{{h}}
   -\frac{T(C)-\phi(U_{j_*-1})}{{h}}\right)
\end{equation*}
that is
\begin{equation*}
 \zeta'(C)=\frac{2(T(C)-C)}{((\phi^+)^{-1}(T(C))-(\phi^-)^{-1}(T(C))){h}}.
\end{equation*}
If $C\in[A,B]$, then $T(C)=C$, and $\zeta'=0$;
if $C>B$, then $T(C)=B$ and $\zeta'<0$; 
finally, if $C<A$, then $T(C)=A$ and $\zeta'>0$.
Thus, the entropy jump conditions are satisfied.  
The transition interface algorithm is defined by
\begin{equation}\label{interfacealgo}
	\zeta^{n+1}= \zeta^{n}+\zeta'(C^n)\,\Delta t
	\qquad\textrm{and}\qquad
	j^{n+1}_*=\left[\frac{\zeta^{n+1}}{{h}}\right]+1.
\end{equation}
Next, let us fix our attention on the scheme \eqref{schemerelaxed} in the
steady interface case (i.e. $C(U^n,j^n_*)\in[A,B]$ for any $n$).  In this
situation, it is possible to consider a corresponding semi-discrete scheme.
  Setting $V=\phi(U)$, relations \eqref{trans2}
becomes
\begin{equation}\label{trans3}
 V_{j_*}=V_{j_*+1}=\frac{1}{2}\left(V_{j_*-1}+V_{j_*+2}\right),
\end{equation}
so that the original system for the variable $V=(V_1,\dots,V_J)$ can be reduced to a
system for $\hat V:=(V_1,\dots,V_{j_*-1}, V_{j_*+2},\dots,V^J)$.  The system
to be satisfied by $\hat V$, taking into account \eqref{trans3}, is
\begin{equation}\label{VeqRel}
 \frac{d\hat V}{ds}=-\hat{\mathbb{A}}\,\hat V,
\end{equation}
where $s=2\tau=2t/{h}^2$ and $\hat{\mathbb{A}}$ is the symmetric $(J-2)\times
(J-2)$ matrix defined by
{
\begin{equation}\label{matrixAhat}
 \hat{\mathbb{A}}:=\left(\begin{array}{cccccccccccccc}
      1 & -1& 0 &\hdots &\hdots &\hdots &\hdots &\hdots &\hdots  &0 &0 &0 \\
      -1 & 2 & -1 &\ddots & \hdots & \hdots  &\hdots&\hdots &\hdots  & 0& 0&0\\
      0 & -1 & 2 & \ddots   & \ddots  & \hdots & \hdots&\hdots &\hdots & 0& 0&0\\
      \vdots & \ddots & \ddots & \ddots &\ddots &\ddots  &&& & \vdots&\vdots &\vdots \\
      \vdots &\vdots & \ddots & -1 & 2 & -1 & \ddots& & & \vdots&\vdots &\vdots  \\
      \vdots & \vdots &\vdots &  \ddots & -1& 3/2 & -1/2 & \ddots& & \vdots &\vdots &\vdots   \\
      \vdots & \vdots & \vdots   & & \ddots & -1/2 & 3/2 & -1 &   \ddots& \vdots  &\vdots &\vdots\\
      \vdots & \vdots & \vdots   &   & &  \ddots & -1 &  2 & -1&   \ddots& \vdots &\vdots \\
      \vdots & \vdots &  \vdots  &  &    &   &  \ddots & \ddots  & \ddots  & \ddots &\ddots&\vdots \\
      0& 0 &0 &\hdots &\hdots &\hdots &\hdots&\ddots &\ddots   &   2    &  -1     & 0 \\
      0& 0 &0 &\hdots &\hdots &\hdots &\hdots&\hdots &\ddots &   -1    &  2     & -1 \\
      0& 0 &0 &\hdots &\hdots &\hdots &\hdots&\hdots &\hdots   &  0     & -1     & 1
    \end{array}\right)
\end{equation}
}(here the lines $(\dots,0,-1, 3/2, -1/2, 0,\dots)$ and $(\dots,0,-1/2,
 3/2, -1, 0,\dots)$ correspond to the indices $j_*-1$ and $j_*+2$).  
 The matrix $\hat{\mathbb{A}}$ defined in \eqref{matrixAhat} has spectral properties similar
 to the ones of the matrix $\mathbb{A}$.  
 Specifically, $\hat{\mathbb{A}}$ is semi-definite positive, $0\in \sigma(\hat{\mathbb{A}})$ and 
 $\hat{\mathbb{A}} \hat r_1=0$ where $\hat r_1=(1,\dots,1)^T/\sqrt{J-2}$.  
 The matrix $\hat P_1:=\hat r_1\otimes \hat r_1={\mathbf 1}/(J-2)$, where ${\mathbf 1}$ 
 is the $(J-2)\times(J-2)$ matrix composed by all ones, is the eigenprojection
 corresponding to the eigenvalue $0$.  
 As a consequence, solution to \eqref{VeqRel} with initial data $V^0$ converges 
 exponentially fast to
\begin{equation*}
 \hat P_1\,V^0=\frac{1}{J-2}\sum_{j=1}^{J-2}\,V^0_j\,(1,\dots,1)^T,
\end{equation*}
that is the solution $U$ converges to a Riemann-shaped steady state as $\tau\to+\infty$ 
as described in \eqref{lim_U}.

\subsection*{Numerical experiments for the two-phase scheme}
Next, we perform numerical experiments for the scheme \eqref{schemerelaxed}
with the same Riemann-type initial data considered for the scheme \eqref{backforD}
specifically with $(u^-,u^+)=(-2,3.5)$ and $(u^-,u^+)=(-2,4)$.
Results are shown in Figure \ref{fig:Lprof_pwls_zero}.
\begin{figure}[ht]
  \centering
  \subfigure[$u^-=-2$, $u^+=3.5$]{\includegraphics[width=7cm,keepaspectratio]{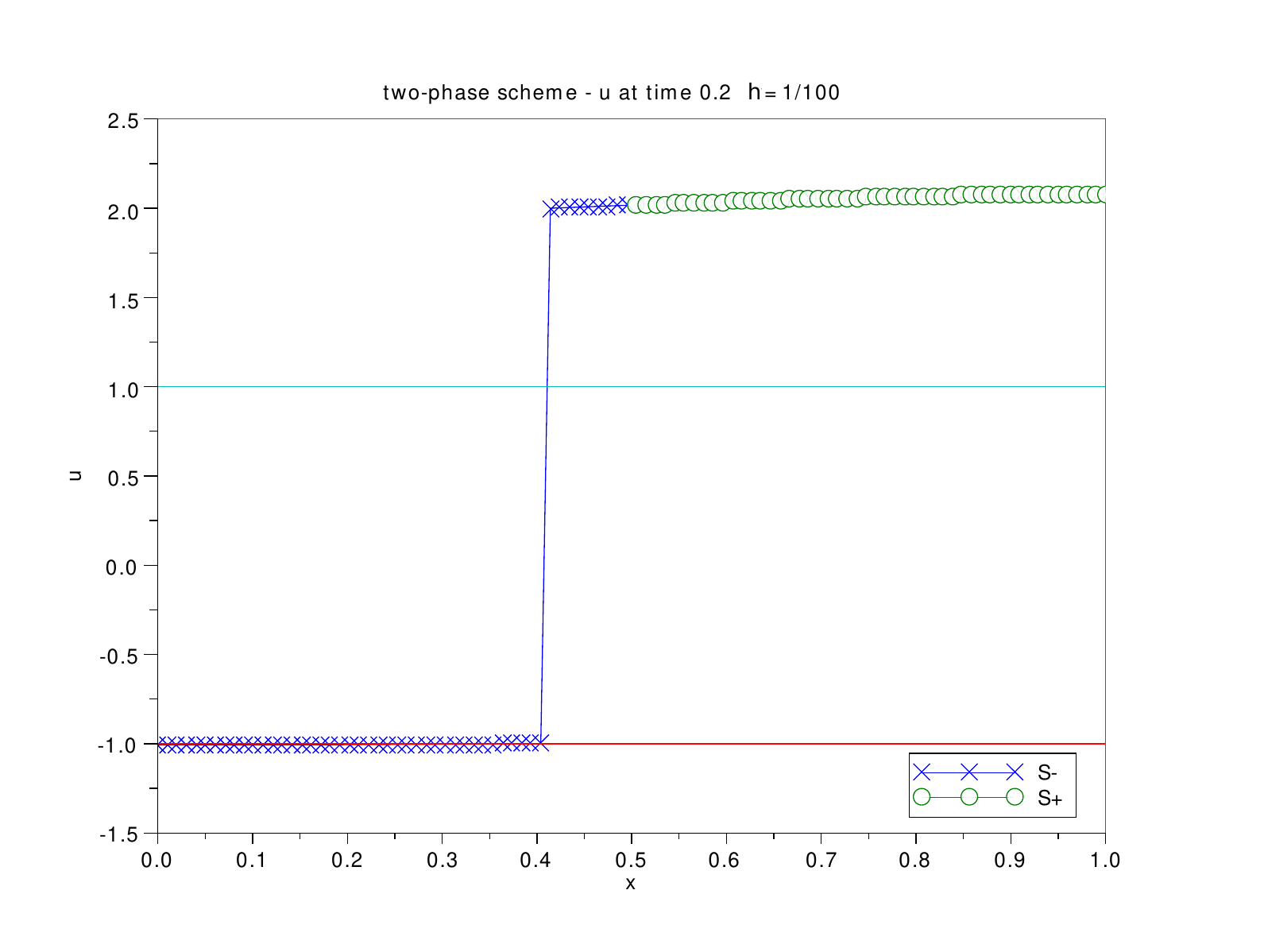}}
\subfigure[$u^-=-2$, $u^+=4$]{\includegraphics[width=7cm,keepaspectratio]{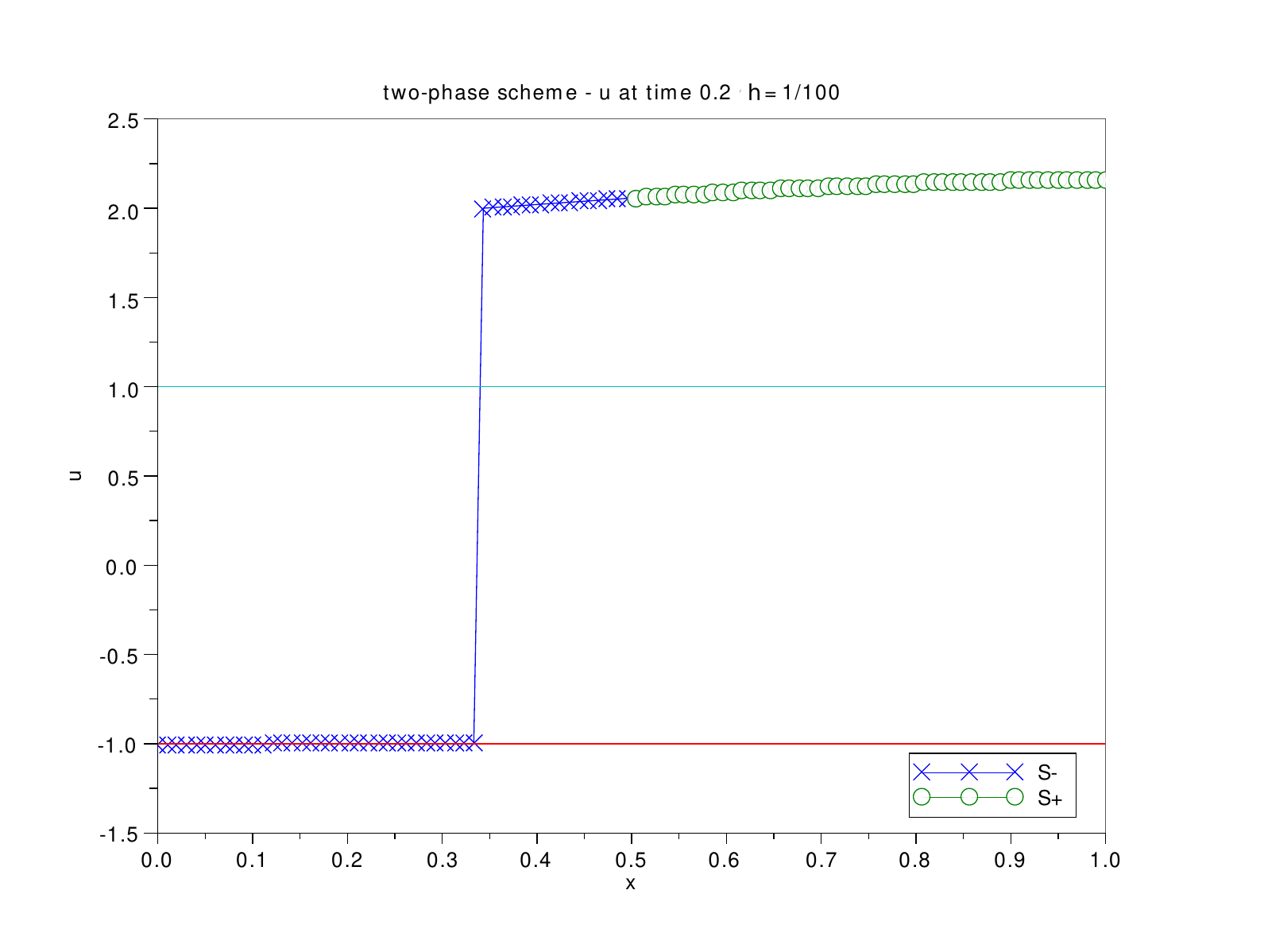}}
\caption{\small\footnotesize Evolution of the Riemann data. 
The two graphs show the profile of $u$ at time $t=0.2$ and ${h}=0.01$.
The choice of different symbols helps keeping track of the initial position of the interface.}
\label{fig:Lprof_pwls_zero}
\end{figure}
   
In term of the solution profiles $u$ (or $\phi(u)$), the results obtained with the present approach
and the one considered in the previous Section are  consistent with each other.  
On the contrary, in term of position of the moving interface, the behavior is very different,  
being much more regular and non oscillatory in the case of the algorithm \eqref{schemerelaxed} 
(compare Figures \ref{fig:mvt_explicit}(b) and \ref{fig:2p_mvt_explicit}).
\begin{figure}[ht]
  \centering
   \subfigure[Movement of the interface]{\includegraphics[width=7cm, height=5cm]{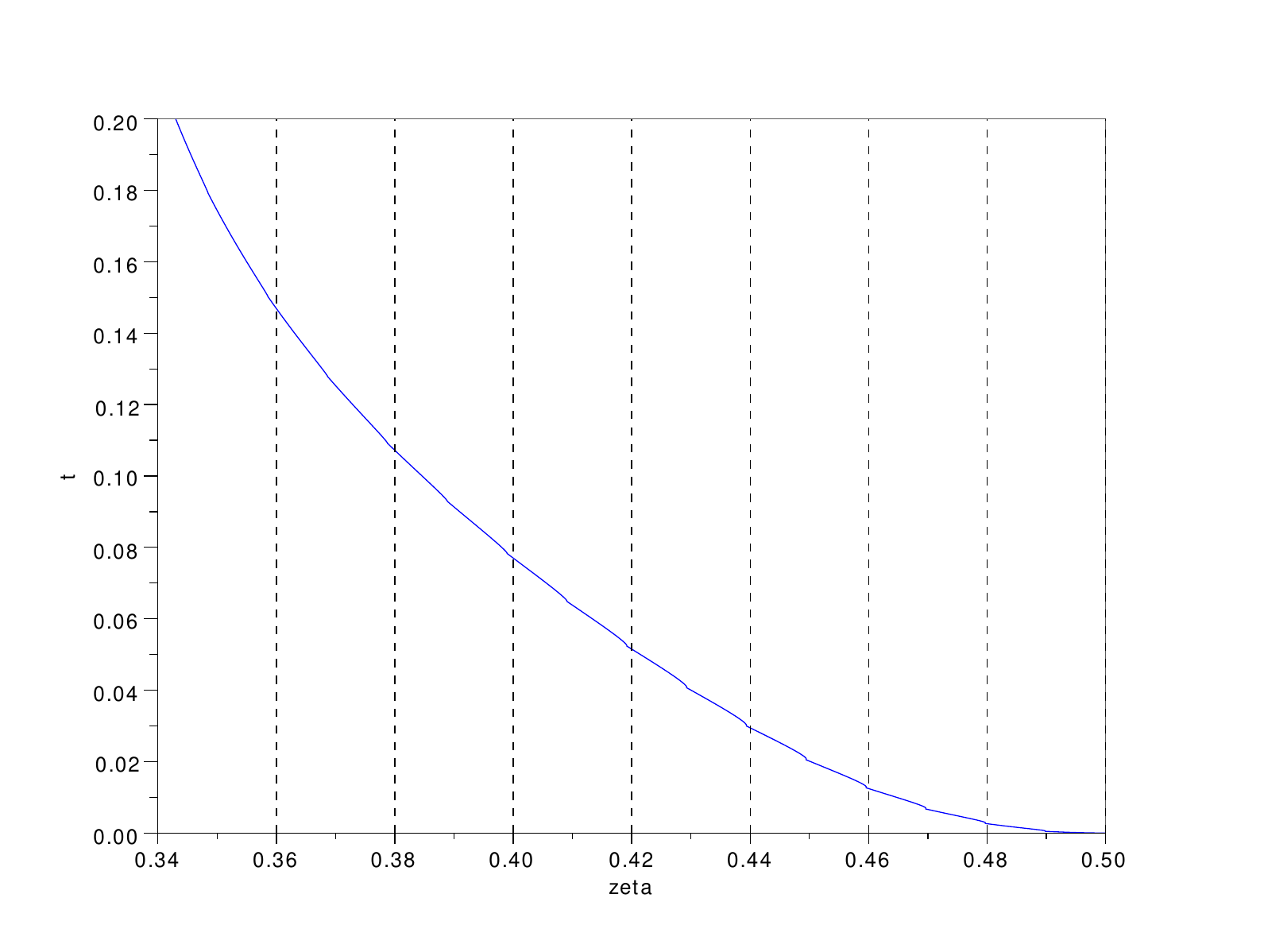}}
   \subfigure[Values of {$-[\partial_x(\phi(u))]/[u]$}]{\includegraphics[width=7cm, height=5cm]{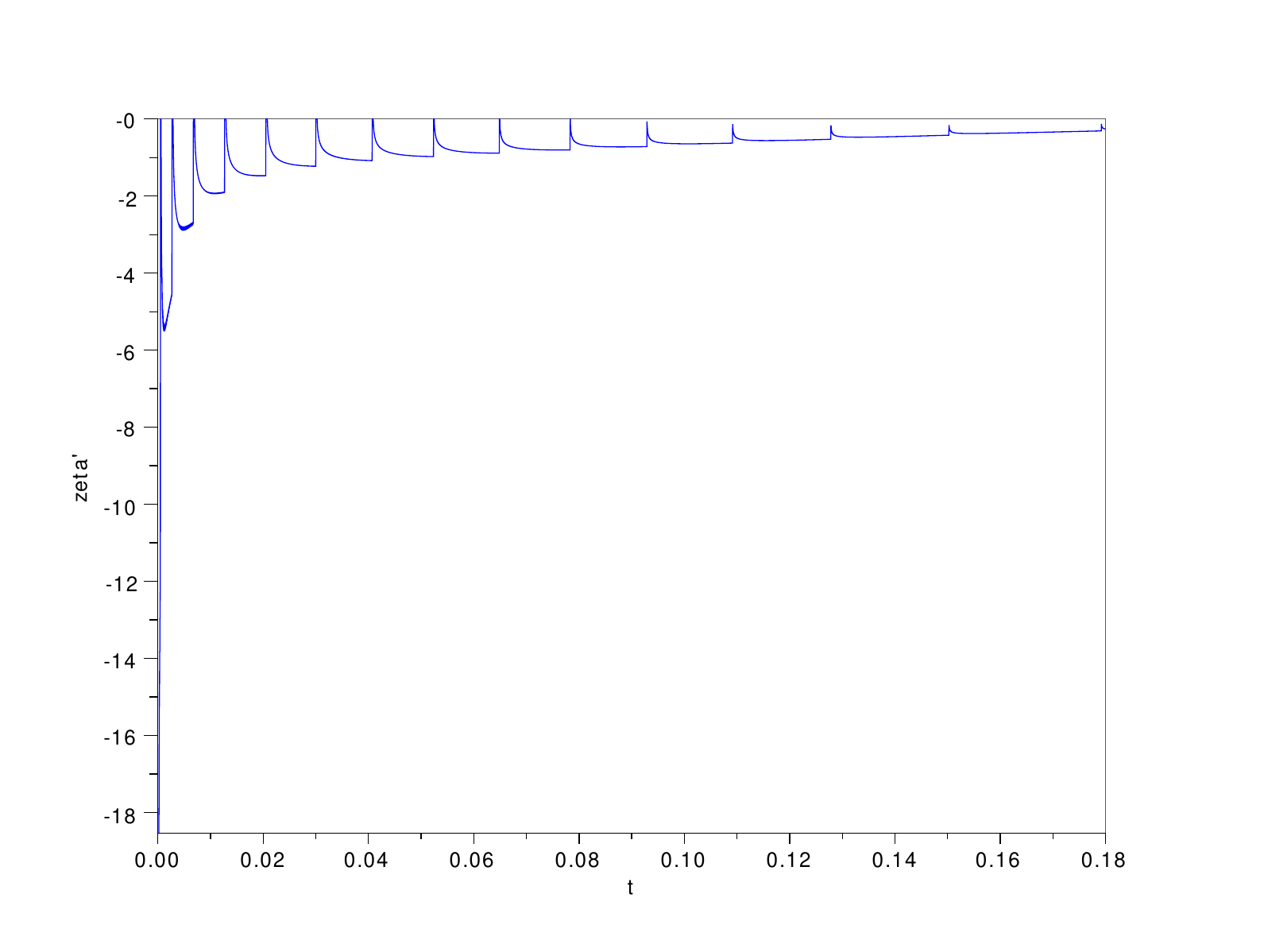}}
  \caption{\small Riemann initial data $u^-=-2$, $u^+=4$ with ${h}=0.01$.}
  \label{fig:2p_mvt_explicit}
\end{figure}

Indeed, the two approaches differ strongly in the determination of the interface position. 
When dealing with the algorithm \eqref{backforD}, the position of the transition interface 
is determined by the last index $j$ such that $u_j\leq b$.  
As a consequence, the movement of the interface is always concentrated at specific iterations 
and have an amount of multiples of the space-discretization size ${h}$.  
The algorithm proposed for the limit model \eqref{nonlinpar} is based on the assumption that both the solution 
$u$ and the interface position $\zeta$ are unknowns of the problem.  
Formula \eqref{interfacealgo} permits to define a smoother location of the transition interface since 
its position is determined by calculating the values $\zeta'$ at each step.  
The discrepancy between the two approaches is shown in Figure \ref{fig:comp_fe_2p}, for the 
Riemann problem with initial data $(-2,4)$.
\begin{figure}[ht]
          \centering \subfigure[Solution of explicit and two-phase schemes]
          	{\includegraphics[width=7cm,keepaspectratio]{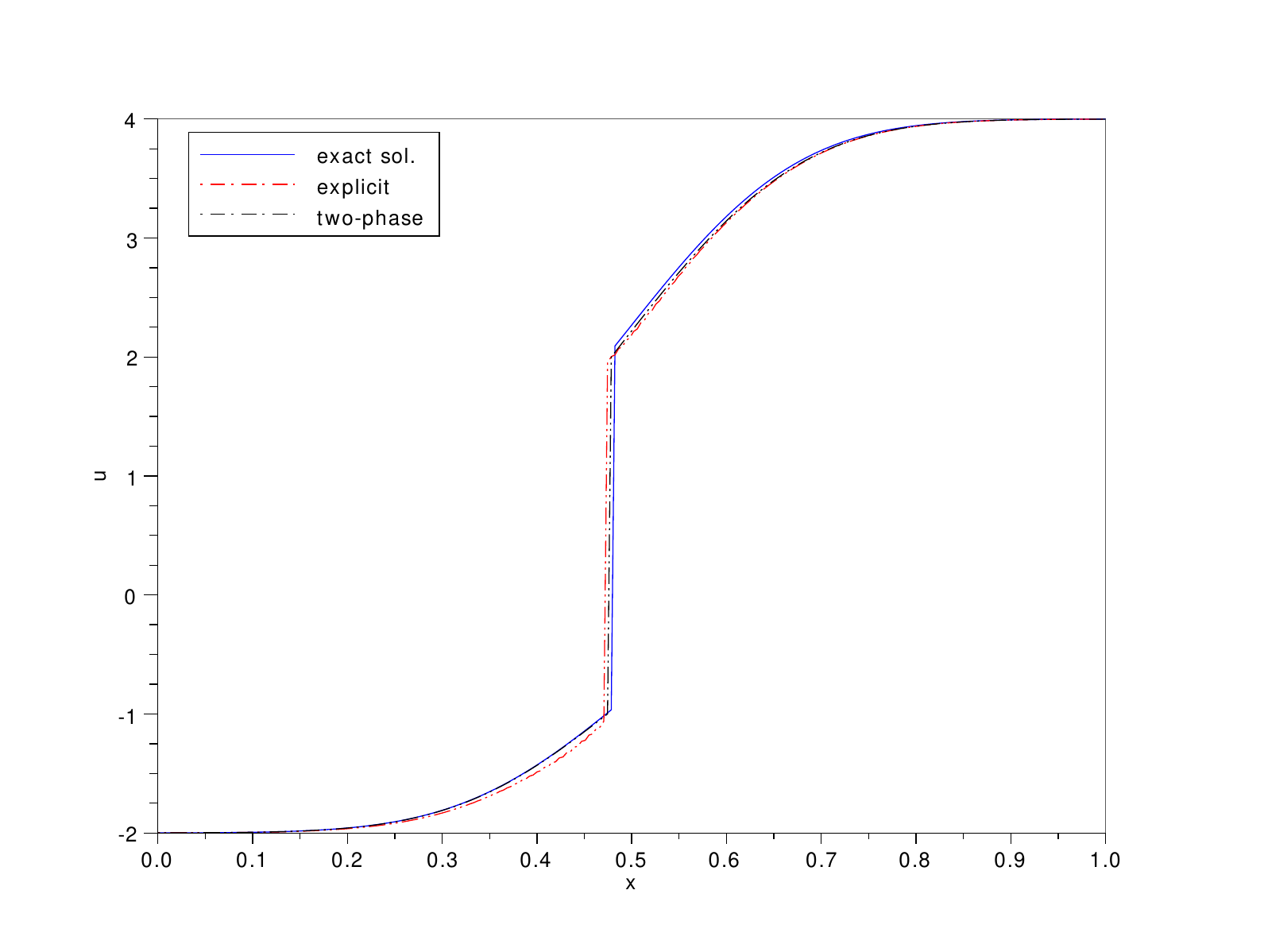}}\qquad
          \subfigure[Zoom]
          	{\includegraphics[width=7cm,keepaspectratio]{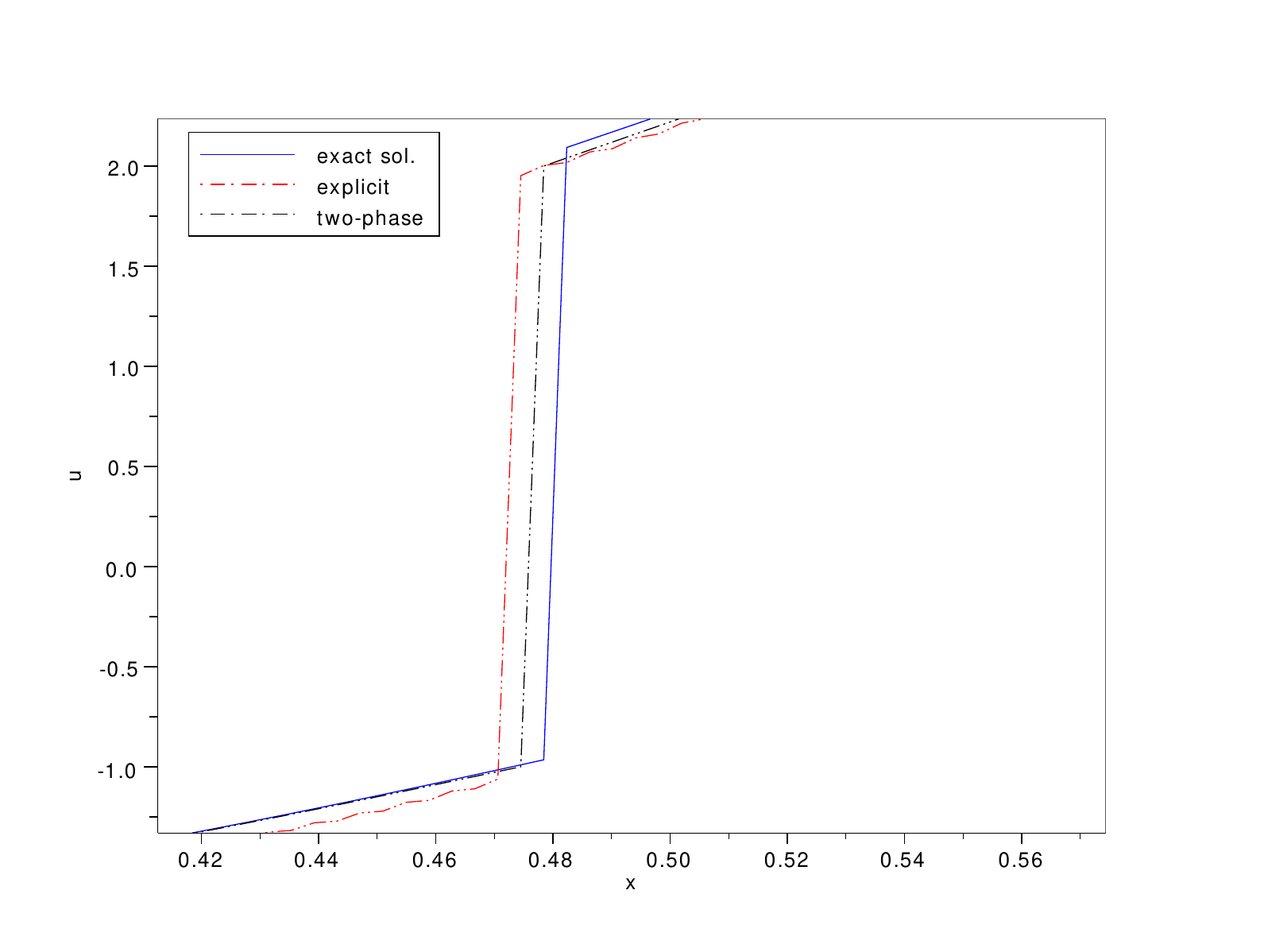}}\\
\caption{\small Comparison of the numerical solutions of the explicit and two-phase schemes 
with the exact solution for $(u^-,u^+)=(-2,4)$ at $t=0.005$ for ${h}=1/255$}
\label{fig:comp_fe_2p}
\end{figure}
The two-phase scheme and the explicit scheme are both of order $1$ in ${h}$ 
(see Fig. \ref{fig:cv}), but the error for the two-phase scheme (compared to the exact 
solution computed in \eqref{decogi}, the constant $\bar{\xi}$ being computed with 
Newton's algorithm) is less (see Figure \ref{fig:comp_fe_2p})
\begin{figure}[ht]
       \centering 
       \subfigure[Two-phase scheme: error as a function of the mesh size ${h}$]
       		{\includegraphics[width=7cm,keepaspectratio]{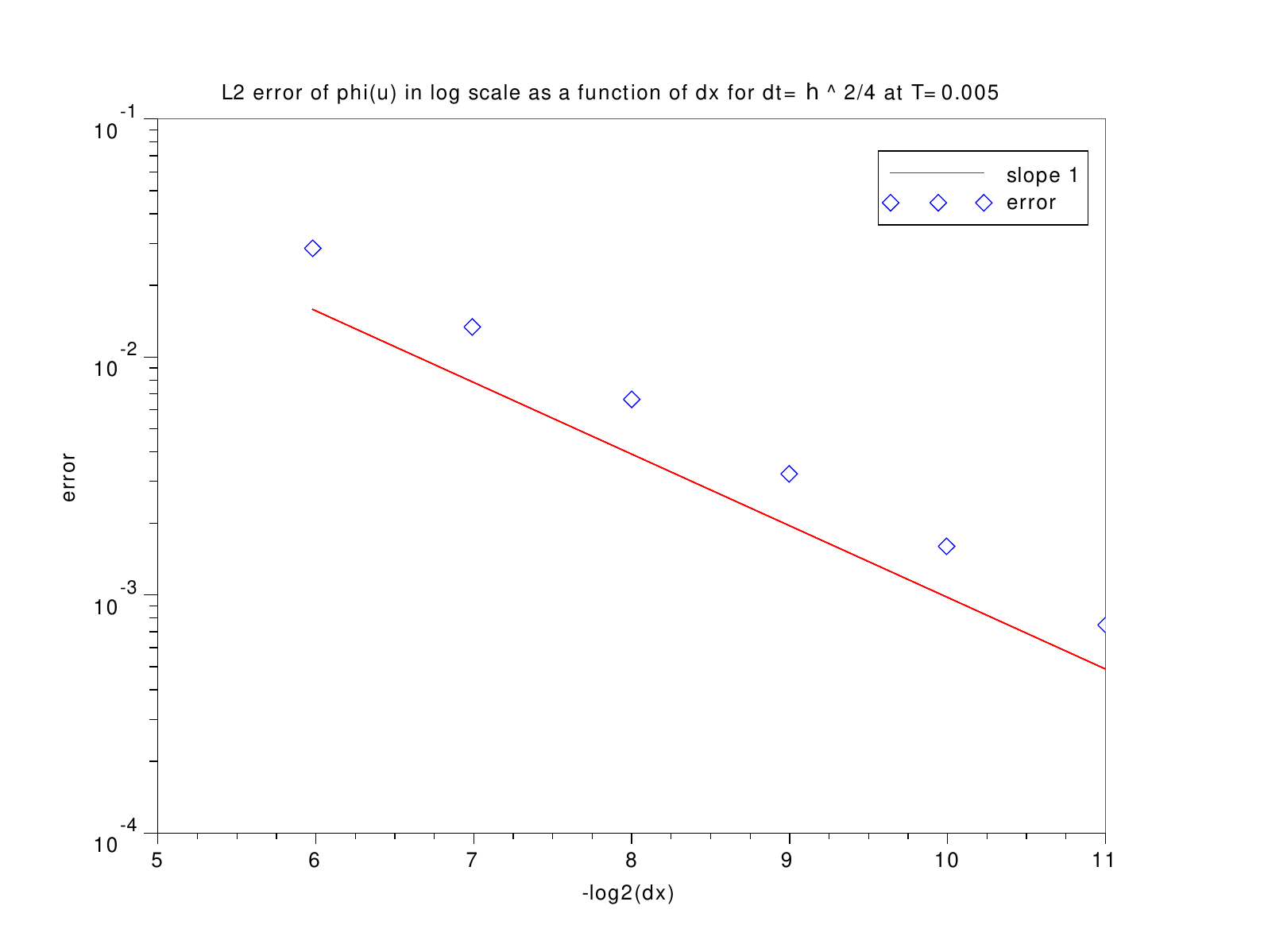}}\qquad
       \subfigure[Explicit scheme:  error as a function of the mesh size ${h}$]
       		{\includegraphics[width=7cm,keepaspectratio]{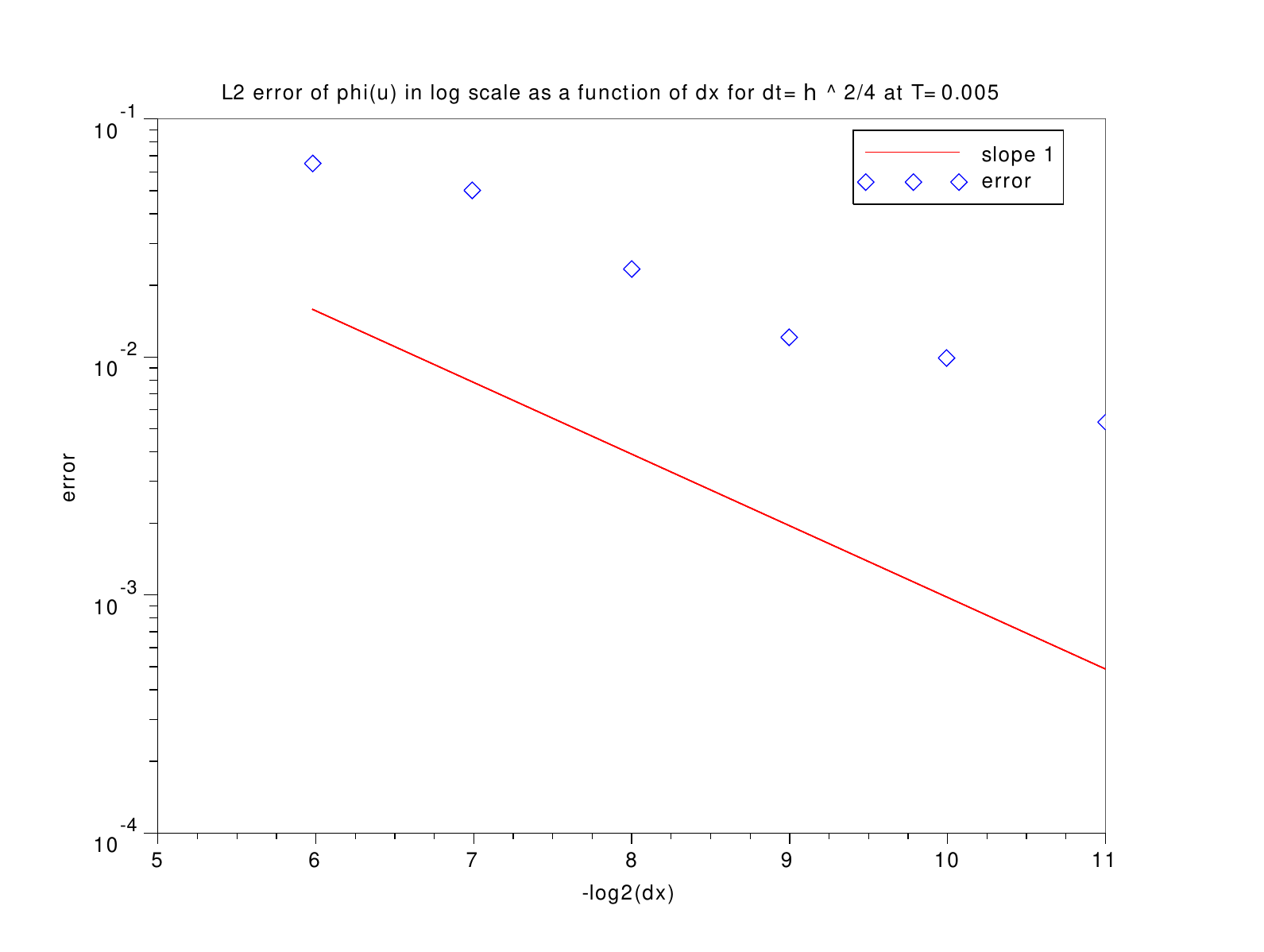}}\\
\caption{\small Convergence error for $\phi(u)$ in $L^2$ norm}
\label{fig:cv}
\end{figure}
and the interface $\zeta$  is better approximated and smoother (see Figure \ref{fig:error_zeta}). 
However, due to the intermediate computations, the numerical cost is also higher than 
the one of the explicit scheme (see Table 2).
\begin{center} 
{\footnotesize\begin{table}[!h]\label{tab:tps_exec_mod}
      \begin{tabular}{|c|c|}\hline
        Two-phase scheme & Explicit\\\hline
 \begin{tabular}{c|c|c}
    $J$  &
   error & cpu\\\hline
   \hline
6&	0.028548	&	0.23	\\
7&	0.0133988	&	0.34	\\
8&	0.0065732	&	1.16	\\
9&	0.0032091	&	7.18	\\
10&	0.0015852	&	64.73	\\
11&	0.0007498	&	919.69	\\
\end{tabular}&
 \begin{tabular}{c|c|c}
    $J$  &
   error & cpu\\\hline
   \hline
6&	0.0600245	&	0.13	\\
7&	0.0300282	&	0.26	\\
8&	0.0150125	&	1.05	\\
9&	0.0075057	&	6.02	\\
10&	0.0037539	&	52.31	\\
11&	0.0016060	&	723.27	\\
\end{tabular}
\\
\hline
\end{tabular}\vskip.25cm
\caption{\small Comparison of $L^2$ errors of $\phi(u)$ and execution times (cpu) for
  the two-phase and explicit schemes for $(u^-,u^+)=(-2,4)$.}
\end{table}}
\end{center}

\section*{Acknowledgments}
The first and the second author are thankful to the Dipartimento di Matematica ``G. Castelnuovo'', 
Sapienza, Universit\`a di Roma (Italy), and the Laboratoire  Paul Painlev\'e, Universit\'e de Sciences et 
Technologies de Lille (France), respectively, for the kind hospitality that eventually led to the present research.

\end{document}